\documentclass[11pt,reqno,tbtags]{article}
\usepackage[margin=1in]{geometry}
\usepackage{amssymb}
\usepackage{amsmath}
\usepackage{graphicx}
\usepackage{setspace}
\usepackage{MnSymbol} 

\begin{document}

\newtheorem{theorem}{Theorem}[section]
\newtheorem{conjecture}[theorem]{Conjecture}
\newtheorem{corollary}[theorem]{Corollary}
\newtheorem{lemma}[theorem]{Lemma}
\newtheorem{claim}[theorem]{Claim}
\newtheorem{proposition}[theorem]{Proposition}
\newtheorem{construction}[theorem]{Construction}
\newtheorem{definition}[theorem]{Definition}
\newtheorem{question}[theorem]{Question}
\newtheorem{problem}[theorem]{Problem}
\newtheorem{remark}[theorem]{Remark}
\newtheorem{observation}[theorem]{Observation}

\newcommand{\ex}{{\mathrm{ex}}}

\newcommand{\EX}{{\mathrm{EX}}}

\newcommand{\AR}{{\mathrm{AR}}}

\def\endproofbox{\hskip 1.3em\hfill\rule{6pt}{6pt}}
\newenvironment{proof}%
{%
\noindent{\it Proof.}
}%
{%
 \quad\hfill\endproofbox\vspace*{2ex}
}
\def\qed{\hskip 1.3em\hfill\rule{6pt}{6pt}}
\def\ce#1{\lceil #1 \rceil}
\def\fl#1{\lfloor #1 \rfloor}
\def\lr{\longrightarrow}
\def\e{\varepsilon}
\def\ex{{\rm\bf ex}}
\def\cA{{\cal A}}
\def\cB{{\cal B}}
\def\cC{{\cal C}}
    \def\cD{{\cal D}}
\def\cF{{\cal F}}
\def\cG{{\cal G}}
\def\cH{{\cal H}}
\def\ck{{\cal K}}
\def\cI{{\cal I}}
\def\cJ{{\cal J}}
\def\cL{{\cal L}}
\def\cM{{\cal M}}
\def\cP{{\cal P}}
\def\cQ{{\cal Q}}
\def\cR{{\cal R}}
\def\cS{{\cal S}}
\def\cE{{\cal E}}
\def\cT{{\cal T}}
\def\ex{{\rm ex}}
\def\pr{{\rm Pr}}
\def\exp{{\rm  exp}}

\def\wt{\widetilde{T}}
\def\bkl{{\cal B}^{(k)}_e}
\def\cmkt{{\cal M}^{(k)}_{t+1}}
\def\cpkl{{\cal P}^{(k)}_e}
\def\cckl{{\cal C}^{(k)}_e}
\def\pkl{\mathbb{P}^{(k)}_e}
\def\ckl{\mathbb{C}^{(k)}_e}

\def\mC{{\cal C}}

\def\imp{\Longrightarrow}
\def\1e{\frac{1}{\e}\log \frac{1}{\e}}
\def\ne{n^{\e}}
\def\rad{ {\rm \, rad}}
\def\equ{\Longleftrightarrow}
\def\pkl{\mathbb{P}^{(k)}_e}

\def\mE{\mathbb{E}}

\def\mP{\mathbb{P}}

\def \e{\varepsilon}

\voffset=-0.5in
	
\setstretch{1.1}
\pagestyle{myheadings}
\markright{{\small \sc  Jiang, Yepremyan:}
  {\it\small Supersaturation of linear even cycles in  linear hypergraphs}}
\newcommand{\brm}[1]{\operatorname{#1}}

\title{\huge\bf  Supersaturation of Even Linear Cycles in Linear Hypergraphs}

\author{
Tao Jiang\thanks{Department of Mathematics, Miami University, Oxford,
OH 45056, USA. E-mail: jiangt@miamioh.edu. Research supported in part 
by National Science Foundation grant DMS-1400249. }
\quad \quad Liana Yepremyan
\thanks{Department of Mathematics, University of Oxford, UK. E-mail:yepremyan@maths.ox.ac.uk. Research supported in part by ERC Consolidator Grant 647678
} }

\date{\today}
\maketitle
\begin{abstract}
A classic result of Erd\H{o}s and, independently, of Bondy and Simonovits \cite{BS} says that the maximum number of edges
in an $n$-vertex graph not containing  $C_{2k}$, the cycle of length $2k$, is $O( n^{1+1/k})$. Simonovits  established
a corresponding supersaturation result for $C_{2k}$'s, showing that there exist positive constants $C,c$ depending only
on $k$ such that every $n$-vertex graph $G$ with $\brm{e}(G)\geq Cn^{1+1/k}$ contains at least $c\left(\frac{\brm{e}(G)}{\brm{v}(G)}\right)^{2k}$
many copies of $C_{2k}$, this  number of copies tightly achieved by the random graph (up to a multiplicative constant).

In this paper, we extend Simonovits' result to a supersaturation result of $r$-uniform linear cycles of even length in $r$-uniform
linear hypergraphs. Our proof is self-contained and includes the $r=2$ case. As an auxiliary tool, we develop a reduction lemma from 
general host graphs to almost-regular host graphs that can be used for other supersaturation problems, and may therefore be of independent interest.
\end{abstract}

\section{Introduction}
 
One of the central problems in extremal graph theory is the Tur\'an problem, where for a fixed graph $H$ and fixed $n$, one wishes to
determine the maximum number of edges  an $n$-vertex graph can have without creating a copy of $H$ as a 
subgraph. This number is called the \emph{Tur\'an number} of $H$ and denoted by $ex(n,H)$. The celebrated Erd\H{o}s-Stone-Simonovits \cite{ES} theorem  says that $ex(n,H)= \left(1-\frac{1}{\chi(H)-1}\right)n^2 +o(n^2)$, where $\chi(H)$ is the chromatic number of the graph $H$. This  solves the Tur\'an problem  asymptotically for all non-bipartite graphs $H$.
However, asymptotic results or exact results  are known only for a handful number of bipartite graphs.  While the Tur\'an problem asks for the threshold on the number of edges on $n$ vertices
that guarantees at least one copy of  $H$, it is natural to ask what is 
the minimum number of copies of $H$ guaranteed in a host graph once its number of edges exceeds $ex(n,H)$.
Such problems are referred to as {\it supersaturation problems}. When $H$ is non-bipartite, 
we know the correct order of magnitude of the answer.
Let $H$ be a graph with $\chi(H)=p\geq 3$ and $\brm{v}(H)$ vertices.
A  simple averaging argument  (see, for example, Lemma 2.1 in \cite{keevash-survey}) can be used to show that 
for any  $\varepsilon>0$ there exist $\delta, n_0>0$ such that 
if $G$ is a graph on $n\geq n_0$ vertices with $\brm{e}(G)\geq (1-\frac{1}{p-1}+\varepsilon)\binom{n}{2}$ then
$G$ contains at least $\delta\binom{n}{\brm{v}(H)}$ copies of $H$. This count is tight up to a multiplicative constant, as  shown by the random graph of the same edge density as $G$. The threshold on the  number of edges on $G$ for which the count is valid is also asymptotically best possible, as shown by the Tur\'an graph $T_{n,p-1}$, which is defined as the balanced blowup of the complete graph on $(p-1)$ vertices. For the supersaturation problem for bipartite graphs,  Erd\H{o}s and Simonovits \cite{ES-cube-supersat} made the following  conjecture in the 1980s.
\begin{conjecture} \label{bipartite-supersaturation} {\bf \cite{ES-cube-supersat}}
Let $H$ be a bipartite graph with $v$ vertices and $e$ edges.
Suppose that  $ex(n,H)=O(n^{2-\alpha})$ for some real $0<\alpha<1$.
Then there exist $\alpha'\leq \alpha$ and  constants $C,c>0$ such that if $G$ is an $n$-vertex graph
with 
\begin{equation}\label{eq:condeg}
\brm{e}(G)\geq Cn^{2-\alpha'}
\end{equation} edges then $G$ contains at least $c\frac{(\brm{e}(G))^e}{n^{2e-v}}$ copies of $H$.
\end{conjecture}

The Erd\H{o}s-R\'enyi random graph $G(n,p)$ with $p=\frac{\brm{e}(G)}{\binom{n}{2}}$ shows that if Conjecture \ref{bipartite-supersaturation}
is true then it is best possible up to a multiplicative constant.  Conjecture~\ref{bipartite-supersaturation} is closely related to the famous Sidorenko's conjecture~\cite{Sidorenko}, which says that if $H$ is any bipartite graph then the random graph with edge density $p$ has in expectation
 the minimum number of homomorphic copies of $H$ over all graphs of the same order and edge density.
For dense enough host graphs $G$, i.e. (if we do not worry about finding the optimal $\alpha'$ in Conjecture~\ref{bipartite-supersaturation})
then any family of graphs that satisfy Sidorenko's conjecture also satisfy Conjecture~\ref{bipartite-supersaturation}. However,
works on Conjecture \ref{bipartite-supersaturation} often aim at finding the best possible threshold beyond which the counting
statement holds. In fact, in the same paper Erd\H{o}s and Simonovits made two even stronger conjectures by relaxing condition \eqref{eq:condeg}
to $\brm{e}(G)\geq C\cdot ex(n,H)$ and to $\brm{e}(G)\geq(1+\varepsilon) ex(n,H)$, respectively.
For details, see \cite{ES-cube-supersat}. At this stage, resolving these  stronger versions seem hopeless since the exact value or even just the order of magnitude of $ex(n,H)$ is only known for very  few bipartite graphs $H$. 

Now we turn our attention to the main focus of this paper, that is, the supersaturation of linear cycles of even length in linear $r$-uniform hypergraphs, or in short, $r$-graphs. First let us give the background for $r=2$.  A classic result of Erd\H{os} (unpublished) and of Bondy and Simonovits \cite{BS} establishes that $ex(n,C_{2k})=O(n^{1+1/k})$.
The explicit upper bound that Bondy and Simonovits gave was  $ex(n,C_{2k})\leq 100kn^{1+1/k}$.
This upper bound was later improved by Verstra\"ete to $8(k-1)n^{1+1/k}$ for sufficiently large $n$, by Pikhurko \cite{pikhurko} to  $(k-1)n^{1+1/k}+O(n)$, 
and by Bukh and Jiang \cite{BJ} to $80\sqrt{k} \log k n^{1+1/k}+O(n)$.
It is conjectured by Erd\H{o}s and Simonovits that $ex(n,C_{2k})=\Omega(n^{1+1/k})$ also holds.
This is known to be true for $k=2,3,5$. 

For supersaturation of even cycles, it was mentioned in~\cite{ES-cube-supersat} that
Simonovits proved Conjecture~\ref{bipartite-supersaturation} with $\alpha=\alpha'=1-1/k$. This proof has not been published at the time, but is expected to appear in an upcoming paper of Faudree and Simonovits \cite{FS-cycle}.
Very recently, Morris and Saxton \cite{MS} developed a balanced version of the supersaturation result for even cycles, 
which they use to obtain a sharp result on the number of $C_{2k}$-free graphs via the container method.
Since Morris and Saxton require a  balanced version  of supersaturation where the collection of $C_{2k}$'s they obtain 
are, informally speaking, uniformly distributed,  their proof is
quite involved. In this paper, we extend Simonovits' supersaturation result of even cycles to supersaturation of even linear cycles in linear $r$-graphs.
Our proof is self-contained and includes the $r=2$ case.

Before stating our main result, we need a few definitions. An $r$-graph $G$ is called {\it linear} if  any two edges share at most one vertex. For instance, all $2$-graphs are linear. The \emph{linear Tur\'an number} of an $r$-graph $H$, denoted by $ex_{l}(n,H)$ is defined to be the  the maximum number of edges  an $n$-vertex linear $r$-graph can have without creating a copy of $H$. The study of linear Tur\'an numbers of linear $r$-graphs is motivated in part by their similarity to  the Tur\'an numbers of $2$-graphs. Also, such studies 
were implicit in some classic extremal hypergraph problems, such as the famous $(6,3)$-problem (see \cite{BES} and \cite{RS}) which is asymptotically 
equivalent to determining the linear Tur\'an number of a linear $3$-cycle. The $(6,3)$-problem asks for the maximum size of an $n$-vertex
$3$-graph in which no six vertices span three or more edges. Note that the usual Tur\'an number $ex(n,H)$ of a linear $r$-graph $H$ and the linear Tur\'an number $ex_l(n,H)$ of $H$
are typically very different. The former is already at least $\binom{n-1}{r-1}$ as long as $H$ contains two disjoint edges
while the latter is $O(n^2)$.

An $r$-uniform {\it linear cycle} $C^{(r)}_m$of length $m$ is obtained from a $2$-uniform $m$-cycle $v_1v_2\dots v_mv_1$
by extending each $v_iv_{i+1}$ (indices taken modulo $m$) with an $(r-2)$-tuple $I_{i}$  such that the tuples $I_i$ are pairwise disjoint for distinct indices. Collier-Cartaino, Graber, and Jiang \cite{CGJ} extended aforementioned result of Bondy and Simonovits on  even cycles in $2$-graphs to linear cycles in linear $r$-graphs.
They showed that for all $r\geq 3$ and $m\geq 4$, $ex_l(n,C^{(r)}_{m})=O(n^{1+1/\lfloor m/2\rfloor})$. 
For even linear cycles, their result also works for uniformity $r=2$. It is interesting to note that when $r\geq 3$
the linear Tur\'an number of an odd linear cycle resembles that of an even linear cycle, which is very different from the situation for $r=2$. 

Our main result is the supersaturation version of the linear Tur\'an result for linear cycles, but only for even linear cycles.

\begin{theorem}\label{thm:main} Given $k, r\geq 2$, there exist  constants $C$, $c$ such that  if $G$ is an $n$-vertex linear $r$-graph with $\brm{e}(G)\geq Cn^{1+1/k}$
then $G$ contains at least $c\left(\frac{\brm{e}(G)}{n}\right)^{2k}$ copies of $C^{(r)}_{2k}$.
\end{theorem} 

It is not hard to see that this lower bound on the number of copies of  linear cycles is tight, up to a multiplicative constant. Indeed, for $r=2$, in the  Erd\H{o}s-R\'enyi graph $G(n,p)$, in expectation there are $\Theta(p^{2k} n^{2k})$ many $2k$-cycles. For $r\geq 3$, one may consider random subgraphs of almost complete partial Steiner systems. An $(n,\lambda, r,q)$-\emph{Steiner system} is defined to be an $r$-graph on $n$ vertices such that every $q$-tuple is in exactly $\lambda$ many $r$-edges.  A \emph{partial $(n,\lambda, r, q)$-Steiner system} is defined to be an $r$-graph on $n$ vertices such that every $q$-tuple is in at most $\lambda$ many $r$-edges.  It was proved by R\"odl~\cite{rodl} that for all  $n$, there are partial $(n,1, r, q)$-Steiner systems with  $(1-o(1))\binom{n}{q}$ edges. (Note that this is also implied by  recent solution of existence conjecture by Keevash~\cite{keevash}, while the $q=2$ case was proved much earlier by Wilson \cite{wilson1, wilson2, wilson3}.) By taking random subgraphs of such partial Steiner systems with $\lambda=1$ and $q=2$, one can  show that for every $n$ and $0\leq e\leq  {n \choose 2} $, there is a linear $r$-graph $G$ on $n$ vertices and   $\brm{e}(G)=e$ in which the number of copies of
the linear cycles of length $2k$ is $O\left(\left(\frac{e}{n}\right)^{2k}\right)$.

Our Theorem \ref{thm:main} includes the $r=2$ case as a special case and has a much simpler proof than the proof of Morris and Saxton of their 
stronger version of supersaturation. We use an approach developed by Faudree and Simonovits \cite{FS} in the study of 
the Tur\'an numbers of  so-called \emph{$\Theta$-graphs} which in its original form is not well-suited for
effective counting of $C_{2k}$'s. So we adapt their approach to facilitate counting.
Our proofs are  greatly simplified via a reduction tool which allows us to reduce the supersaturation problem in a general host $r$-graph to one which has some regularity property. Our regularization tool is an analogue of a regularization theorem of Erd\H{o}s and Simonovits for the Tur\'an problem and can be used for supersaturation problems of more general graphs.  The exact statement of our reduction lemma for linear $r$-graphs is slightly technical, so we refer the reader to Section~\ref{sec:reduction} for the precise statement (see Theorem~\ref{thm:reduction}). 

We organize the rest of the paper roughly as follows. In Section~\ref{sec:notation} we present notation and definitions.
In Section~\ref{sec:reduction} we develop our reduction results, which as we mentioned, may be of independent interest.
In Section~\ref{sec:graphs} and Section~\ref{sec:hypergraphs} we give proofs of the $r=2$  and $r\geq 3$ cases of Theorem~\ref{thm:main}, respectively. Even though we could have proved Theorem \ref{thm:main} for the general $r$ directly, we feel that proving the $r=2$ case first helps illustrating the main ideas. However, we will give two slightly different proofs for $r=2$ and $r\geq 3$, modulo the reduction mentioned earlier. Both of these proofs could be written for  all $r\geq 2$. The proof we present for the $r=2$ case is more constructive and
gives a better bound on constants. The proof we present for $r\geq 3$ follows the approach of Faudree-Simonovits 
more closely (modulo the reduction arguments) and is perhaps more intuitive and easier to follow for some readers.  In Section~\ref{sec:conclusion} we give some concluding remarks.


\section{Notation and Definitions} \label{sec:notation}

For an $r$-graph $G$, we let $\brm{v}(G)$ and $\brm{e}(G)$ denote its number of vertices and edges, respectively.
Let $\Delta(G), \delta(G)$ denote the maximum and minimum degree of $G$, respectively. 
Given a vertex $v\in V(G)$, the {\it link } of $v$ in $G$, denoted by $L_G(v)$, is defined to be
\[L_G(v)=\{ I\in [V(G)]^{(r-1)}|  I\cup \{v\}\in E(G)\},\]
where recall that $[V(G)]^{(r-1)}$ refers to the family of all $(r-1)$-subsets of $V(G)$. 

 Given a real $q\geq 1$, we say that  an $r$-graph $G$ is {\it $q$-almost-regular} if $\Delta(G)\leq q\delta(G)$ holds. 
 Given positive constants $C,\gamma$, we say that an $r$-graph $G$ is $(C,\gamma)$-dense if $\brm{e}(G)\geq C(\brm{v}(G))^\gamma$. 

 We  denote by \emph{$t_H(G)$} the number of copies of an $r$-graph $H$ in an $r$-graph $G$. In this paper we will always assume that $H$ has no isolated vertices, however, it is not hard to see that all the mentioned results  work for  all graphs.  Given a $2$-graph $F$, the {\it $r$-expansion} of $F$ is the $r$-graph obtained by replacing each edge $e$ of $F$
with $e\cup I_e$, where $I_e$ is an $(r-2)$-tuple of new vertices, such that the $I_e$'s are pairwise disjoint for distinct edges $e$.  Note that $F^{(2)}=F$.  If $G$ is an $r$-expansion of a $2$-graph $F$, then we call $F$ a \emph{skeleton} of $G$. So, for example, the linear $r$-cycle, $C_m^{(r)}$ is  the $r$-expansion of the $2$-uniform $m$-cycle. Now we define the notion of supersaturation of expansions in linear $r$-graphs. 
\begin{definition} \label{def:supersat}
\label{def:supersat}Given a $2$-graph $F$ with $v$ vertices and $e$ edges, $r\geq 2$ and $c$  a positive real, we say that a linear $r$-graph $G$ \emph{$c$-supersaturates}  $F^{(r)}$, if $t_{F^{(r)}} (G)\geq c \frac{(\brm{e}(G))^e}{(\brm{v}(G))^{2e-v}}$.
\end{definition}

Note that  for $r=2$ this definition is the usual supersaturation for $2$-graphs ( as in Conjecture~\ref{bipartite-supersaturation}), that is, a graph $G$ $c$-supersaturates another graph $F$ with $e$ edges and $v$ vertices if  $t_{F} (G)\geq c \frac{(\brm{e}(G))^e}{(\brm{v}(G))^{2e-v}}$.  As we have discussed earlier, for $2$-graphs the bound  on the number of copies of $F^{(r)}=F$   in Definition~\ref{def:supersat} is  achieved up to a multiplicative constant by the random graph of the same edge density.  For general $r\geq 3$, the bound  is tight as well, obtained  by random subgraphs of appropriate edge density in almost complete partial Steiner systems, which as we discussed in the introduction exist.
\begin{proposition} \label{prop:sat-lower}
Let $F$ be a $2$-uniform graph with $v$ vertices and $e$ edges and let $r\geq 2$. For all $n$ and all $0<E\leq \binom{n}{2}$
there exist $n$-vertex linear $r$-graphs $G$ with $\brm{e}(G)=E$ in which the number of copies of
$F^{(r)}$ is $O( \frac{E^e}{n^{2e-v}})$. 
\end{proposition}
The proof of Proposition \ref{prop:sat-lower} is folklore, so we omit it. One can take a random subgraph of an almost complete Steiner systems and apply standard concentration inequalities.

Let us remark that the bound given in Definition~\ref{def:supersat} is 
specific to the setting where the host graph is linear and embedded graphs is an expansion.
In a different setting, the supersaturation problem typically becomes very different and the expected  optimal count  
is expected to be different.  In fact, as mentioned in the introduction, the thresholds for forcing even just one copy of $F^{(r)}$
can be very different depending on whether we require the host graph to be linear or not.

Recall that an $r$-graph $G$ is {\it $r$-partite} with an $r$-partition $A_1,\dots, A_r$ if each edge of $G$ contains exactly one vertex from each $A_i$.  Given an $r$-partite $r$-graph $G$ with an $r$-partition $A_1,\dots, A_r$, we define a {\it $2$-projection} of $G$ to be the $2$-graph we obtain by taking the projection of the edges of $G$ 
onto two of its partition classes. More formally, for $1\leq i<j\leq r$, we define the $(i,j)$-{\it projection} of $G$, denoted by $P_{i,j}(G)$, 
to be  a $2$-graph whose edge set is defined as
\[E(P_{i,j}(G))=\{e \cap (A_i\cup A_j) |\, e\in E(G)\}.\]

Note that when $G$ is  linear, for any $1\leq i<j\leq r$ the map $e\rightarrow e\cap{(A_i\cup A_j)}$ is a bijection, and in particular,  $\brm{e}(G)=e(P_{i,j}(G))$. Next, we give  the following slightly technical definition of what we call \emph{projection-restricted supersaturation} in linear $r$-partite  $r$-graphs. 
In the next section we show that if we have this form of supersaturation for an expansion $F^{(r)}$ in  linear $r$-graphs that have an almost-regular $2$-projection, 
then we also get supersaturation of $F^{(r)}$ in all linear $r$-graphs. 

\begin{definition} \label{def:proj-supersat}
Given a $2$-graph $F$ with $v$ vertices and $e$ edges, $r\geq 2$ and $c$  a positive real. For a linear, $r$-partite $r$-graph $G$  and any $2$-projection of it, call $P$, we say  $(G,P)$ $c$-supersaturates $F^{(r)}$  if $$t_{F^{(r)}} (G)\geq c \frac{(\brm{e}(P))^e}{(\brm{v}(P))^{2e-v}}.$$
\end{definition}

Note that for $2$-graphs, Definition \ref{def:proj-supersat} coincides with Definition \ref{def:supersat}.

\section{Reduction results}\label{sec:reduction}

 Erd\H{o}s and Simonovits \cite{ES-almost-regular} proved the following ``regularization" theorem for $2$-graphs.

 \begin{theorem}{\rm \cite{ES-almost-regular}} \label{ES-almost-regular}
 Let $0<\alpha<1$ be a real and $q=20\cdot 2^{(1/\alpha)^2}$. There exists $n_0=n_0(\alpha)$ such that if $G$ is a $(1, 1+\alpha)$-dense  graph on $n\geq n_0$ vertices then there exists a $q$-almost-regular subgraph of $G$, say $G'$, which is  $(2/5, 1+\alpha)$-dense such that $v(G')>n^{\alpha\frac{1-\alpha}{1+\alpha}}$.
 \end{theorem} 

Theorem \ref{ES-almost-regular} is a useful tool for the Tur\'an problem for $H$. Indeed, given a dense enough $G$, we may first find an almost-regular subgraph $G'$ that has similar density 
as $G$ and look for a copy of  $H$ in $G'$. This theorem itself is not sufficient for establishing supersaturation results for $2$-graphs
since we look to force many copies of $H$ in $G$. By going into $G'$, we might lose many copies of $H$. What we likely need is the existence
of a collection of dense enough almost regular subgraphs of $G$ which together supply the number of copies of $H$ that we need.
Indeed, this is the rough idea behind the following lemma, which may be viewed as some kind of extension of Theorem \ref{ES-almost-regular}. For any $s,t$ integers, $t\geq s\geq 1$, and a graph $H$, we define $f(H, s, t)=\frac{(\brm{e}(H))^s}{(\brm{v}(H))^t}$.

\begin{lemma}\label{lem:decomposition}
Let $\alpha$ be a real and $s,t$ integers, where $0<\alpha<1$ and $t\geq s\geq 1$, then there exist positive  reals $C_0= C_0(\alpha,s,t)$ and $q=q_{\ref{lem:decomposition}}(\alpha,s,t)$ such that the following holds. For every $C\geq C_0$ if $G$ is a $(C,1+\alpha)$-dense graph $G$ then it
contains a collection of edge-disjoint subgraphs $G_1,\dots, G_m$ 
satisfying
\begin{enumerate}
\item $\forall i\in [m]$, $G_i$ is $q$-almost-regular and $(\frac{1}{4}C,1+\alpha)$-dense,
\item $\sum_{i=1}^m f(G_i,s,t)\geq \frac{1}{4^s} f(G,s,t)$.
\end{enumerate}
\end{lemma}
\begin{proof}
While we specify the choice of $q$ explicity, we don't do so for $C_0$. We assume $C_0$ is sufficiently large
as a function of $\alpha$, $s$, and $t$. Let $p=\lceil 2^{\max\{\frac{4}\alpha, \frac{2s+t}{t-s+1}\}}\rceil$ and $q=8p$.
By the definition of $p$, we have
\begin{equation} \label{q-conditions}
p^\alpha\geq 16 \mbox{ and } p^{t-s+1}\geq 2^{2s+t}.
\end{equation}

Suppose $G$ has $n$ vertices. 
Let us partition $V(G)$ into $p$ sets $A_1,\dots, A_p$ of almost equal sizes, i.e. each of size $\lceil n/p\rceil$
or $\lfloor n/p \rfloor$, 
such  that $A_1$ contains vertices of the highest degrees in $G$. For convenience, we will drop the
ceilings and floors in our arguments as doing so does not affect the arguments except for the slight changes
to constants. 

We now prove our statement by induction on $n$. When $n<q$ the claim holds trivially
since either $G$ itself is $q$-almost-regular or no $(C,1+\alpha)$-dense
graph on $n<q$ many vertices exists. For the induction step, we consider two cases.

\medskip

{\bf Case 1.} The number of edges in $G$ with at least one endpoint in $A_1$ is at most $\frac{\brm{e}(G)}{2}$.

\medskip
Let $d=d(G)$ be the average degree of $G$.
By our definition of $A_1$, for each vertex $v\in V(G) \setminus A_1$,
we have $d_G(v)\leq pd$; otherwise
$\sum_{u\in A_1} d_G(u)>pd (n/p)=nd$, a contradiction.
Let $G'=G-A_1$. Then 
\[\Delta(G')\leq pd.\]
and $\brm{e}(G')\geq \frac{\brm{e}(G)}{2}$, by initial assumptions. By iteratively deleting vertices whose degree becomes less than $\frac{d}{8}$, we obtain a subgraph
$G''\subseteq G'$ with $e(G'')\geq \brm{e}(G')-\frac{nd}{8}\geq \frac{\brm{e}(G)}{4}$ and $\delta(G'')\geq \frac{d}{8}$.

Since $\Delta(G'')\leq \Delta(G')\leq pd$ and $\delta(G'')\geq d/8$, $G''$ is $8p$-almost-regular,
that is, $G''$ is $q$-almost-regular.
Also, 
\[e(G'')\geq\frac{1}{4}\brm{e}(G)\geq  \frac{1}{4}C [\brm{v}(G)]^{1+\alpha}\geq \frac{1}{4}C[v(G'')]^{1+\alpha}.\]
Thus, $G''$ is $(\frac{1}{4}C,1+\alpha)$-dense. 
Now,
\[f(G'')=\frac{[e(G'')]^s}{[v(G'')]^t}\geq \frac{[\brm{e}(G)/4]^s}{[\brm{v}(G)]^t}\geq \frac{1}{4^s} \frac{[\brm{e}(G)]^s}{[\brm{v}(G)]^t}=\frac{1}{4^s}f(G).\]
So the claim holds by letting our collection of subgraphs be $\{G''\}$.

\medskip

{\bf Case 2.}  The number of edges in $G$ with at least one endpoint in $A_1$ is more than $\frac{1}{2}\brm{e}(G)$.

\medskip

For each $i=2,\dots, p$, let $G_i=G[A_1\cup A_i]$, $n_i=\brm{v}(G_i)$ and $e_i=\brm{e}(G_i)$. Then 
for each $i\in \{2,\dots, p\}$, $n_i=\frac{2n}{p}$. Also, $\sum_{i=2}^p e_i\geq \frac{\brm{e}(G)}{2}$.
Let $\cI=\{2,\dots, p\}$.
Define
\[\cI_1=\left\{i\in \cI: e_i\geq Cn_i^{1+\alpha} \right\} \quad \mbox { and } \quad \cI_2=\cI\setminus \cI_1.\]

Recall that  $p^\alpha\geq 16$.
By the definition of $\cI_2$ and the fact that $n_i=2n/p$ for each $i\in [p]$, we have
\[\sum_{i\in \cI_2} e_i\leq \frac{C}{p^{1+\alpha}} \sum_{i\in \cI_2} (2n)^{1+\alpha} 
\leq \frac{C|\cI_2|2^{1+\alpha}n^{1+\alpha}}{p^{1+\alpha}}
\leq \frac{4Cn^{1+\alpha}}{p^\alpha}\leq \frac{C}{4}n^{1+\alpha}
\leq \frac{\brm{e}(H)}{4}=\frac{\brm{e}(G)}{4}.\]  

Hence,
\begin{equation} \label{ei-sum-graph}
\sum_{i\in \cI_1} e_i\geq \frac{\brm{e}(G)}{4}.
\end{equation}
For each $i\in \cI_1$ since $e_i\geq Cn_i^\gamma$ and $n_i<n$, by the induction hypothesis,
$G_i$ contains a collection of edge-disjoint subgraphs $G_i^1,\dots, G_i^{m_i}$ each of which is
$q$-almost-regular and $(\frac{1}{4}C,1+\alpha)$-dense such that 
\[\sum_{j=1}^{m_i} f(G_i^j)\geq \frac{1}{4^s} f(G_i).\]
Hence, 
\begin{equation}\label{f-sum}
\sum_{i\in \cI_1} \sum_{j=1}^{m_i} f(G_i^j)\geq \frac{1}{4^s}\sum_{i \in \cI_1} f(G_i)=\frac{1}{4^s} \sum_{i\in\cI_1} \frac{e_i^s}{n_i^t}
=\frac{p^t}{4^s(2n)^t} \sum_{i\in \cI_1} {e_i^s}.
\end{equation}
Hence, by \eqref{ei-sum-graph},  \eqref{f-sum}, $p^{t-s+1}\geq 2^{2s+t}$, and the convexity of the function $x^s$, we have
\[\sum_{i\in \cI_1} \sum_{j=1}^{m_i} f(G_i^j)\geq \frac{p^t}{4^s(2n)^t} \frac{(\sum_{i\in \cI_1} e_i)^s}{|\cI_1|^{s-1}}
\geq \frac{p^{t-s+1}}{4^{2s}2^{t}} \frac{e^s}{n^t}\geq \frac{1}{4^s}f(G).\]
Hence the claims holds by letting $\{G_i^j: i\in \cI_1, 1\leq j\leq m_i\}$ be our collection of subgraphs of $G$.
This completes Case 2 and the proof. 
\end{proof}

Now we are ready to state our main result of the section, which informally says that if projection-restricted supersaturation holds for those host graphs which have almost-regular projections, then the supersaturation holds for all graphs. The formal  statement follows.

\begin{theorem}\label{thm:reduction} Let $\alpha\in (0,1)$ be a real and $r\geq 2$. Let $F$ be a graph with $v$ vertices and $e$ edges, where $e\geq v$.
There exists a real $q=q(\alpha, F)\geq 1$ such that the following holds.
Suppose $C,c>0$ are constants such that for every linear $r$-partite $r$-graph $G$ that has a $(C,1+\alpha)$-dense and $q$-almost-regular $2$-projection $P$,  $(G,P)$ $c$-supersaturates $F^{(r)}$. Then  there exist $C', c'$ such that  every linear $r$-partite $r$-graph  that is $(C', 1+\alpha)$-dense $c'$-supersaturates $F^{(r)}$.
\end{theorem}

\begin{proof}  Let $s=e, t=2e-v$. By our assumption, $t\geq s\geq 1$.
We show that  the theorem holds for
$q$ to be chosen as $q_{\ref{lem:decomposition}}(\alpha,s,t)$, derived from Lemma~\ref{lem:decomposition} applied with constant $\alpha,s$ and $t$. Finally let $C'= 4C$, $c'= \frac{c r^{2e-v}}{2^{4e-v}}$. Let $G$ be an $r$-partite $r$-graph on $n$ vertices such that $\brm{e}(G)\geq 4C n^{1+\alpha}$.  Suppose $G$ has an $r$-partition $(A_1,A_2,\dots, A_r)$, such that $|A_1|\geq |A_2|\geq\dots |A_r|$. It follows that $|A_1\cup A_2|\geq 2n/r$, and $H=P_{1,2}(G)$ is a $(4C,1+\alpha)$-dense graph.
By Lemma \ref{lem:decomposition}, there exists a collection of 
edge-disjoint subgraphs $H_1,\dots, H_m$ of $H$, each of which is $(C, 1+\alpha)$-dense
and $q$-almost-regular, such that 
\[\sum_{i=1}^m f(H_i,s, t )\geq \frac{1}{4^s} f(H,s,t).\]
For each $i\in [m]$, let $G_i$ be the subgraph of $G$ such that $P_{1,2}(G_i)=H_i$.
For each $i\in [m]$, since $H_i$ is $(C,1+\alpha)$-dense and $q$-almost-regular, by the hypothesis of the theorem, $(G_i,H_i)$ $c$-supersaturates $F^{(r)}$.
That is,
\[t_{F^{(r)}}(G_i)\geq c\frac{(e(H_i))^s}{(v(H_i))^t}= c f(H_i,s,t).\]
Since the $H_i$'s are edge-disjoint, $G_i'$'s are also edge-disjoint (i.e. there is no edge contained in two different $G_i$'s). Thus, we have

\[t_{F^{(r)}}(G)\geq \sum_{i=1}^m t_F(G_i)\geq c\sum_{i=1}^m f(H_i)\geq \frac{c}{4^s} f(H)=
\frac{c}{4^e} \cdot \frac{(\brm{e}(H))^e}{(\brm{v}(H))^{2e-v}} \geq \frac{c r^{2e-v}}{2^{4e-v}}\cdot \frac{(\brm{e}(G))^e}{n^{2e-v}}.\]
\end{proof}

Theorem \ref{thm:reduction} says that to establish supersaturation of $F^{(r)}$ in an $n$-vertex linear $r$-partite $r$-graph $G$, we may assume $G$ has a dense enough almost-regular $2$-projection $P$. Our next lemma can be used to show that we may further assume $P$ to have edge density exactly
$\Theta(\brm{v}(P)^{1+\alpha})$, where $\alpha$  is any fixed real for which $ex(n,F^{(r)}) = O(n^{1+\alpha})$. The proof uses random sampling and the classic Chernoff bound, which we state here for completeness.
\begin{lemma} { (f.e. \cite{JLR} Corollary 2.3)} \label{lem:chernoff}
Given a binomially distributed variable  $X\in \brm{BIN}(n,p)$ we have $\mP(|X-E(X)|\geq a E(X))\leq 2e^{-\frac{a^2}{3}E(X)}$, as long as $0<a\leq 3/2$.
\end{lemma}

\begin{lemma} \label{lem:random-sample} Let $r\geq 2$ be an integer. 
 Let $\alpha\in (0,1)$ be a real. Let $F$ be a graph with $v$ vertices and $e$ edges, where $e\geq v$.
There is a constant $m_0=m_0(\alpha)$ such that the following holds for all $M\geq m_0$.
Suppose $D,q,c>0$ are reals where $q\geq 1$ such that
for every linear $r$-partite $r$-graph $G'$ that has  a $2$-projection $P'$ on $m\geq M$ vertices satisfying
$Dm^{\alpha} \leq \delta(P')\leq \Delta(P')\leq  3qDm^{\alpha}$
we have that $(G',P')$ $c$-supersaturates $F^{(r)}$. Then  for every
linear $r$-partite $r$-graph $G$ that has  a $(qD,1+\alpha)$-dense and $q$-almost-regular $2$-projection $P$ on at least $M$ vertices, 
$(G,P)$ $\frac{c}{2^{e+1}}$-supersaturates $F^{(r)}$. 
\end{lemma}

\begin{proof} The choice of $m_0$ will be given implicitly in the proof. Let $G$ be a linear $r$-partite graph with an $r$-partition 
$(A_1,A_2, \dots, A_r)$, where without loss of  generality $P=P_{1,2}(G)$ is $(qD,1+\alpha)$-dense, $q$-almost-regular, and has $m=\brm{v}(P)\geq M$ vertices. 
The idea is to sample randomly a subgraph $G'$ of $G$ with an appropriate edge probability,
count $F^{(r)}$ in $G'$, and then use it to bound the count of $F^{(r)}$ in $G$.
Now let $\delta(P)$, $\Delta(P)$ and $d$ denote the minimum,maximum and average degrees in $P$, respectively. 
Then $d=\frac{2 \brm{e}(P)}{\brm{v}(P)}\geq \frac{2qDm^{1+\alpha}}{m}=2qD m^{\alpha}$. 
Since $P$ is $q$-almost-regular, we have $\delta(P)\geq d/q \geq 2Dm^{\alpha}$.  

Let $p=\frac{2Dm^{\alpha}}{\delta(P)}$ and let $G'$ be a random subgraph of $G$, obtained by including each edge of $G$
in $G'$ independently with probability $p$. 
 Then $$\mE[\brm{e}(G')] = p \brm{e}(G) \textit{ and } \mE[t_{F^{(r)}}(G')]=t_{F^{(r)}}(G)\cdot p^{e} \textit{ and }\forall v\in V(G') \, \mE[d_{G'}(v)] = pd_G(v).$$
Since $G$ is linear and $P$ is a $2$-projection, for each $v\in V(P)=A_1\cup A_2$ we have $d_G(v)=d_P(v)$.
Since $P$ is $q$-almost-regular, $\Delta(P)/\delta(P)\leq q$.  So, for each $v\in A_1\cup A_2$, we have 
\[\mE[d_{G'}(v)] = pd_G(v) =pd_P(v)\leq \frac{2Dm^{\alpha}}{\delta(P)} \Delta(P) \leq 2qDm^{\alpha},\]
and similarly $\mE[d_{G'}(v)]  \geq 2Dm^{\alpha}$. 

Now  random variables $d_{G'}(v)$ and $\brm{e}(G')$ have binomial distributions. Hence using Markov's inequality and Chernoff's inequality, one can show that
$$\mP[t_{F^{(r)}}(G')>2t_{F^{(r)}}(G)\cdot p^{e}]<\frac{1}{2}, $$
$$\mP[\brm{e}(G')<\frac{p}{2} \brm{e}(G)]<\frac{1}{4}$$
and $$\mP[\exists v \in A_1\cup A_2  \textit{ such that } d_{G'}(v)<Dm^{\alpha} \textit{ or } d_{G'}(v)>3q D m^{\alpha}]<1/4.$$
In some of the inequalities above, we used Chernoff (and the union bound for the last one). The desired inequalities 
hold when $m$ is large enough as a function of $\alpha$ which is guaranteed by choosing $m_0$ to be large enough.
So there exists a subgraph $G'$ of $G$ satisfying 
\begin{equation} \label{count-transfer}
\brm{e}(G')\geq \frac{p}{2}\brm{e}(G), t_{F^{(r)}}(G')\leq 2t_{F^{(r)}}(G)\cdot p^{e}
\end{equation} 
and  that for each $v\in A_1\cup A_2$
\begin{equation} \label{A12-degree}
Dm^{\alpha}\leq d_{G'}(v) \leq 3q D m^{\alpha}.
\end{equation}

Now, let $P'=P_{1,2}(G')$. Since there is a bijection between $E(G')$ and $E(P')$, for each $v\in V(P)$,
we have $d_{P'}(v)=d_{G'}(v)$. By \eqref{A12-degree}, for each $v\in V(P')=A_1\cup A_2$, we have
\[ Dm^{\alpha} \leq d_{P'}(v) \leq 3qDm^{\alpha}.\] 
 Thus, by the hypothesis of our theorem, $(G', P')$ $c'$-supersaturates $F^{(r)}$. By \eqref{count-transfer},
 we have 
\[t_{F^{(r)}}(G) \geq \frac{1}{2p^e}t_{F^{(r)}}(G')  \geq \frac{c}{2p^e} \cdot \frac{(\brm{e}(G'))^e}{m^{2e-v}} \geq  \frac{c}{2p^e} \cdot \frac{p^e(\brm{e}(G))^e}{2^e m^{2e-v}}  =   
\frac{c}{2^{e+1}}\cdot \frac{(\brm{e}(G))^e}{m^{2e-v}}. \]
\end{proof}

Applying Theorem~\ref{thm:reduction} an Lemma~\ref{lem:random-sample} we obtain the following reduction tool for proving supersaturation of expansions.
In this paper, we use it on $C^{(r)}_{2k}$. But it can be applied to other expansions as well and thus is of independent interest.

\begin{corollary}\label{cor:reduction} 
Let $r\geq 2$ be an integer and $\alpha\in (0,1)$ a real. 
Let $F$ be a graph with $v$ vertices and $e$ edges, where $e\geq v$.
Let  $q=q_{\ref{thm:reduction} }(\alpha,F)$ be given as in Theorem \ref{thm:reduction} and $m_0=m_0(\alpha)$ be given
as in Lemma \ref{lem:random-sample}.
Suppose there exist reals $D,\lambda,M, c>0$, where $\lambda\geq 3q, M\geq m_0$,
such that
for every linear $r$-partite $r$-graph $G$ that has a $2$-projection $P$ on $m\geq M$ vertices
satisfying $Dm^{\alpha} \leq \delta(P)\leq \Delta(P)\leq \lambda Dm^{\alpha}$
we have that $(G,P)$ $c$-supersaturates $F^{(r)}$. Then  there exist $C', c'$ such that  every $(C', 1+\alpha)$-dense linear $r$-graph $G$ $c'$-supersaturates $F^{(r)}$.
\end{corollary}
\begin{proof}
Suppose there exist reals $D,\lambda>0$, where $\lambda\geq 3q$,
such that
for every linear $r$-partite $r$-graph $G$ that has a $2$-projection $P$ on $m\geq M$ vertices
satisfying $Dm^{\alpha} \leq \delta(P)\leq \Delta(P)\leq \lambda Dm^{\alpha}$
 we have that $(G,P)$ $c$-supersaturates $F^{(r)}$.  By Lemma \ref{lem:random-sample},
there exists a constant $c_{\ref{lem:random-sample}}$ such that for every linear $r$-partite $r$-graph $G$
that has a $(qD,1+\alpha)$-dense $q$-almost-regular $2$-projection $P$ on at least $M$ vertices
$(G,P)$  $c_{\ref{lem:random-sample}}$-supersaturates $F^{(r)}$.
Set $C=\max\{qD, M\}$.  Then for every linear $r$-partite $r$-graph $G$ that has a $(C,1+\alpha)$-dense
$q$-almost-regular $2$-projection $P$ we have that $(G,P)$  $c_{\ref{lem:random-sample}}$-supersaturates $F^{(r)}$
(as $P$ must have at least $C\geq M$ vertices).
Applying Theorem \ref{thm:reduction} with the $C$ above and $c=c_{\ref{lem:random-sample}}$,
there exist constant $C'_{\ref{thm:reduction}}$ and $c'_{\ref{thm:reduction}}$ such that every linear
$r$-partite $r$-graph that is $(C'_{\ref{thm:reduction}},1+\alpha)$-dense $c'_{\ref{thm:reduction}}$-supersaturates
$F^{(r)}$.

Let $C'=\frac{r^r}{r!} C'_{\ref{thm:reduction}}$.
Let $G$ be a linear $(C',1+\alpha)$-dense $r$-graph. By a well-known fact, $G$ contains a subgraph $G'$
with $\brm{e}(G')\geq \frac{r!}{r^r} \brm{e}(G)$. Clearly, $G'$ is $(C'_{\ref{thm:reduction}},1+\alpha)$-dense. 
By our discussion above, $G'$ $c'_{\ref{thm:reduction}}$-supersaturates $F^{(r)}$. Hence
$G$ $c'_{\ref{thm:reduction}}$-supersaturates $F^{(r)}$. 
\end{proof}

Applying Corollary \ref{cor:reduction} to even linear cycles, we get

\begin{corollary}\label{cor:linear-cycle-reduction}
Let $r,k\geq 2$ be integers. There exist constants $m_k, q_k$ depending only on $k$ such that
the following holds. Suppose there are reals $D, \lambda,M,c>0$, where $\lambda\geq q_k D$
and $M\geq m_k$, such that
for every $n$-vertex linear $r$-partite $r$-graph $G$ that has a $2$-projection $P$ on $m\geq M$ vertices
satisfying $Dm^{1/k}\leq \delta(P)\leq \Delta (P)\leq \lambda D m^{1/k}$ we have $t_{C^{(2k)}}(G)\geq cm^2$.
Then there exist constants $C',c'$ such that every $n$-vertex linear $r$-graph $G$ with $\brm{e}(G)\geq C'n^{1+1/k}$
satisfies $t_{C^{(r)}_{2k}}(G)\geq c'(\frac{\brm{e}(G)}{n})^{2k}$. \qed
\end{corollary}



\section{Supersaturation of even cycles in graphs}\label{sec:graphs}


\subsection{Preliminary Lemmas}

In this short section, we prove a few lemmas. The first simple lemma
provides a new ingredient to the usual Faudree-Simonovits approach \cite{FS} which
enables more efficient counting of $C_{2k}$'s in the host graph and gives better constants.
On surface, applying it will appear to be essentially equivalent to using the Faudree-Simonovits
method directly. However, in the concluding remarks we will point out some subtle differences to illuminate
the advantage that comes with applying the lemma.

\begin{lemma}\label{balanced-root}
Let $T$ be a tree of height $h$ with a  root  $x$. For each $v\in V(T)$, let $T_v$ be the subtree
of $T$ rooted at $v$. Let $b$ be a positive integer. Let $S$ be a set of at least $bh+1$ vertices in $T$.
Then there exists a vertex $y$ at distance $i$ from $x$, for some $0\leq i\leq h-1$,
such that $|V(T_y)\cap S|\geq |S|-ib$ and that for any child $z$ of $y$ in $T$, $|V(T_z)\cap S|\leq |V(T_y)\cap S|-b$.
\end{lemma}
\begin{proof}
We define a sequence of vertices as follows. Let $x_0=x$. Among all the children of $x_0$, let $x_1$ be one 
such that $T_{x_1}$ contains the maximum number of vertices in $S$. Among all the children of $x_1$, let $x_2$ be one 
that contains the maximum number of vertices in $S$, and etc. Suppose the sequence we define this way is $x_0,x_1,\dots, x_p$, where $p\leq h$ and 
$|V(T_{x_p})\cap S|=1$.
Since $|V(T_{x_0})\cap S|\geq bh+1$ and $|V(T_{x_p})\cap S|=1$, there must exist a smallest index $0\leq i<p$ such that $|V(T_{x_{i+1}})\cap S|\leq 
|V(T_{x_i})\cap S|-b$.
Let $y=x_i$. Then $y$ satisfies the claim.
\end{proof}

\begin{lemma}\label{lem:highmindegree}
Given a bipartite graph $G$ with a bipartition $(A,B)$. There exists a subgraph $G'$ of $G$ with a bipartition $(A',B')$ where
$A'\subseteq A, B'\subseteq B$ such that $\brm{e}(G')\geq \frac{1}{2}\brm{e}(G)$ and that
$\delta_{A'}(G')\geq \frac{1}{4} d_A(G)$ and $\delta_{B'}(G')\geq \frac{1}{4} d_B(G)$.
\end{lemma}
\begin{proof}
Let $d_A=d_A(G)$ and $d_B=d_B(G)$.
Let us iteratively delete any vertex in $A$ whose degree becomes less than $\frac{1}{4}d_A$ and any vertex
in $B$ whose degree becomes less than $\frac{1}{4}d_B$. We continue until we no longer have such vertices
or run out of vertices. Let $G'$ denote the remaining graph. Let $A'$ denote the set of remaining vertices in $A$
and $B'$ the set of remaining vertices in $B$. The number of edges removed in the process is at most
\[\frac{1}{4}d_A|A|+\frac{1}{4} d_B|B|\leq \frac{1}{4} \brm{e}(G)+\frac{1}{4} \brm{e}(G)=\frac{\brm{e}(G)}{2}.\]
Hence, $G'$ is non-empty. By the procedure, each vertex in $A'$ has degree at least $\frac{1}{4} d_A$
and each vertex in $B'$ has degree at least $\frac{1}{4}d_B$ in $G'$. 
\end{proof}

We also need the following crude bound on the number of paths of a given length in an asymmetric bipartite graph. Even though sharper estimates exist in literature,
the lemma suffices for our purposes and is self-contained.

\begin{lemma} \label{lem:path-count}
Let $p$ be a positive integer.
Let $G$ be a bipartite graph with a bipartition $(A,B)$. Let $d_A, d_B$ denote the average degrees of vertices in $A$ and in $B$, respectively.
Suppose $d_A, d_B\geq 8p$. Then the number of paths of length $2p+1$ in $G$ is at least $\frac{1}{2^{6p+1}}\brm{e}(G)(d_Ad_B)^p$.
\end{lemma}
\begin{proof}
By Lemma \ref{lem:highmindegree}, $G$ contains a subgraph $G'$ with a bipartition $(A',B')$ such that $\brm{e}(G')\geq \frac{1}{2} \brm{e}(G)$
and that $d_{A'}\geq \frac{1}{4} d_A, d_{B'}\geq \frac{1}{4} d_B$. Considering growing a $(2p-1)$-path $v_1v_2\dots v_{2p-1}$ where
$v_1\in A, v_{2p-1}\in B$. There are $\brm{e}(G')$ ways to pick $v_1v_2$. Then there are at least $\prod_{i=1}^p (d_{B'}-i)(d_{A'}-i)$ ways
to pick the remaining vertices one by one. Since $d_{A'}\geq \frac{1}{4} d_A\geq 2p$ and $d_{B'}\geq\frac{1}{4}d_B\geq 2p$,
we have $d_{A'}-i\geq \frac{1}{2} d_{A'}\geq \frac{1}{8} d_A$ and $d_{B'}\geq \frac{1}{2}d_{B'}\geq \frac{1}{8} d_B$ for all $i\in [p]$.
Hence the number of $(2p+1)$-paths in $G$ is at least $\frac{1}{2} \brm{e}(G)\cdot \frac{1}{2^{6p}} (d_Ad_B)^p=\frac{1}{2^{6p+1}}\brm{e}(G)(d_Ad_B)^p$.
\end{proof}


\subsection{Proof of Theorem~\ref{thm:main} for $r=2$}

In this section, we reprove the supersaturation result for even cycles in graphs.
That is, we prove the $r=2$ case of Theorem \ref{thm:main}. Here is an outline of the proof.
First we develop a lemma that says if the leaves of a rooted tree $T$ of height at most $k$ are attached
to one partite set of a bipartite graph with average degrees on each side being a large enough constant
depending on $k$, then we get many $C_{2k}$'s that contain vertices in some particular level of $T$.
The key ingredient to proving this lemma is Lemma \ref{balanced-root}.
We then apply this lemma to show that every $n$-vertex almost regular graph $G$ with $\brm{e}(G)=\Theta(n^{1+1/k})$
(with adequate bounds on the coefficients)
has $\Omega(n^2)$ many $C_{2k}$'s. By Corollary~\ref{cor:linear-cycle-reduction} this proves the $r=2$ case of Theorem \ref{thm:main}.

\begin{lemma} \label{lem:c2k-count}
Let $h,k$ be positive integers where $h\leq k$.
Let $G$ be a graph.
Let $T$ be a tree of height $h$ in $G$ with a root $x$. For each $i\in [h]$ let $L_i$ be the set of
vertices at distance $i$ from $x$ in $T$.  Let $W$ be a set of vertices in $V(G)\setminus V(T)$.
Let $F$ denote the bipartite subgraph of $G$ containing all the edges of $G$ between $L_h$ and $W$.
Let $d_L, d_W$ be the average degree in $F$ of vertices in $L_h$ and in $W$, respectively. 
Suppose $d_L, d_W\geq 16k^2$. Then there exists $j\in [h]$ such that the number of $C_{2k}$'s in $G$ that contain
some vertex in $L_j$ is at least $\alpha_k |L_h| d_L^{k-h+j}d_W^{k-h+j-1}$, where $\alpha_k=\frac{1}{2^{6k}}(\frac{1}{2k})^{2k-2}$.
\end{lemma}
\begin{proof}
For each vertex $y$ in $T$, let $T_y$ denote the subtree of $T$ rooted at $y$. We clean up $F$ to get a subgraph
$F'$ of $F$ as follows. First, we delete vertices $w$ in $W$ with $d_F(w)\leq kh$.
Let $W'$ denote the set of remaining vertices in $W$. Let $w\in W'$.
Applying Lemma \ref{balanced-root} to $T$ and $S=N_F(w)$, we conclude that there exists
some vertex $r(w)\in L_j$ for some $j\in [h-1]$ such that there are at least
$|N_F(w)|-kj$ members of $N_F(w)$ that lie in $T_{r(w)}$. Furthermore, for any child $z$ of $r(w)$ in $T$,
there are at least $k$ members of $N_F(w)\cap V(T_{r(w)})$ that lie outside $T_z$.
To form $F'$, we include edges between $w$ and $N_F(w)\cap V(T_{r(w)})$ for each $w\in W'$. By our assumptions,
in forming $F'$ from $F$ we have deleted at most $kh$ edges incident to each $w\in W$. Hence,
\[ e(F')\geq e(F)-kh|W|.\]

For each $j\in [h-1]$, let $W_j=\{w\in W': r(w)\in L_j\}$ and let $F_j$ be the subgraph of $F'$ induced by $L_h\cup W_j$.
Let us choose an $j\in [h-1]$ such that $e(F_j)$ is maximum. Then
\begin{equation} \label{Fj-lower}
e(F_j)\geq\frac{1}{k} e(F')\geq  \frac{1}{k}(e(F)-kh|W|)\geq \frac{1}{2k}e(F),
\end{equation}
where the last inequality follows from the fact that $e(F)=d_W|W|\geq  16k^2|W|\geq 2kh|W|$,

Suppose $L_j=\{z_1,z_2,\dots, z_t\}$. For each $e\in [t]$, let $S_e=L_h\cap V(T_{z_e})$. By our definition of $F'$ and
$F_j$, in $F_j$ each $w\in W_j$ has edges to precisely one $S_e$. Let $N_e =N_{F_j}(S_e)$. Then
$N_1,\dots, N_t$ partition $W_j$ and $F_j$ is the vertex disjoint union of $F_j[S_e\cup N_e]$, for $e\in [t]$.
Let $m=k-(h-j)$. 

\begin{claim} \label{claim:cycle-level}
Every $(2m-1)$-path $P$ in $F_j$ extends to a $C_{2k}$ in $G$ that contains a vertex in $L_j$.
\end{claim}
{\it Proof of Claim.}
Consider any $(2m-1)$-path $P$ in $F_j$. By our discussion, $P\subseteq F_j[S_e\cup N_e]$
for some $e\in [t]$.  Suppose $Q=v_1v_2\dots v_{2m}$, where $v_1\in S_e$ and $v_{2m}\in N_e$.
Then $r(v_2)=r(v_{2m})=z_e$. Let $a$ denote the child of $z_e$ in $T_{z_e}$ such that $v_1$ lies under $a$.
Since $r(v_{2m})=z_e$, by definition, $v_{2m}$ has at least $k$ neighbors in $T_{z_e}$ that lies outside $T_a$.
Among them, at least one, say $u$ lies outside $V(P)$.  Let $Q,Q'$ denote the unique
$(v_1,z_e)$-path and the unique $(u,z_e)$-path in $T_{z_e}$, respectively. Since $v_1$ and $u$ lie under different children of $z_e$,
$V(Q)\cap V(Q')=\{z_e\}$. Now, $P\cup v_{2m}u\cup Q\cup Q'$ is a cycle of length $2m-1+1+2(k-m)=2k$ in $G$ that contains $z_e$.
\qed

\medskip

By Claim \ref{claim:cycle-level} the number of $C_{2k}$'s in $G$ that contain a vertex in $L_j$ is at
least the number of $(2m-1)$-paths in $F_j$.  To complete our proof, it suffices to find
a corresponding lower bound on the number of $(2m-1)$-paths in $F_j$.
For convenience, let $A=L_h$ and $B=W_j$. Let $d_A, d_B$ denote the average degrees in $F_j$ of vertices 
in $A$ and $B$, respectively. By \eqref{Fj-lower},
\[d_A\geq \frac{1}{2k} d_L\geq 8k\geq 8m\quad \mbox { and }\quad  d_B\geq \frac{1}{2k} d_W\geq 8k\geq 8m.\]

By Lemma \ref{lem:path-count} with $p=m-1$, the number of $(2m-1)$-paths in $F_j$ is at least
\[\frac{1}{2^{6(m-1)+1}} e(F) [d_Ad_B]^{m-1}\geq \frac{1}{2^{6m}}(\frac{1}{2k})^{2m-2} e(F) d_L^{m-1}d_W^{m-1}\geq\alpha_k |L_h|d_L^m d_W^{m-1}
=\alpha_k|L_h|d_L^{k-h+j} d_W^{k-h+j-1},\]
where $\alpha=\frac{1}{2^{6k}}(\frac{1}{2k})^{2k-2}$.
This completes our proof. 
\end{proof}

In the next theorem, we use Lemma \ref{lem:c2k-count} to quickly obtain the desired lower bound on the number of
$C_{2k}$'s in almost regular $n$-vertex graphs whose number of edges is $\Theta(n^{1+1/k})$.

\begin{theorem}\label{thm:almost-regular-count}
Let $k\geq 2$ be an integer. Let $D,\lambda>0$ be constants where $D\geq 64k^2$ and $\lambda\geq 1$. 
Let $n_0=(8\lambda)^k$. Let $G$ be an $n$-vertex graph, $n\geq n_0$ such that for each $v\in V(G)$, $D n^{1/k}\leq d(v)\leq \lambda D n^{1/k}$.
Then there exists a positive constant $\beta=\beta(D,\lambda,k)$ such that $t_{C_{2k}}(G)\geq \beta n^2$.
\end{theorem}
\begin{proof} 
For each $x\in V(G)$, let $L_i(x)$ denote the set of vertices at distance $i$ from $x$.
Let $h(x)$ be the minimum $i\leq k-1$ such that $|L_{i+1}(x)|/|L_i(x)|<n^{1/k}$. Clearly $h(x)$ exists
or else we run out of vertices. Let $h=h(x)$. Let $T$ be a breadth first search tree rooted at $x$
that includes $L_0(x), L_1(x),\dots, L_h(x)$. By our assumption, 
\[|L_h(x)|\geq n^{h/k}  \mbox{ and } |L_{h+1}(x)|<n^{1/k}|L_h(x)|.\]
Recall that $|L_{i+1}(x)|/|L_i(x)|\geq n^{1/k}$ for all $i=0,1,\dots, h-1$. Since $n\geq n_0\geq (8\lambda)^k$,
$n^{1/k}\geq 8\lambda$. By our assumption,
\[|V(T)\setminus L_h(x)|\leq |L_h|\sum_{i=1}^{h-1}(\frac{1}{8\lambda})^i\leq \frac{1}{4\lambda} |L_h(x)|.\]
Let $F$ be the bipartite subgraph of $G$ consisting of all the edges of $G$ between $L_h(x)$ and $L_{h+1}(x)$.
The total number of edges of $G$ incident to $L_h(x)$ is at least $Dn^{1/k}|L_h(x)|/2$. 
Among them, the number of edges that are incident to $V(T)\setminus L_h(x)$ is at most 
\[(1/4\lambda ) \lambda D n^{1/k}|L_h(x)|=(1/4)Dn^{1/k}|L_h(x)|.\]
Hence,
\[e(F)\geq (1/4)Dn^{1/k}|L_h(x)|.\]
Let $d_A, d_B$ denote the average degrees in $F$ of vertices in $L_h(x)$ and $L_{h+1}(x)$, respectively.
Then
\[d_A\geq (1/4)Dn^{1/k}\geq 16k^2n^{1/k}\geq 16k^2,\]
and 
\[d_B\geq (1/4)Dn^{1/k}|L_h(x)|/|L_{h+1}(x)|\geq (1/4)D\geq 16k^2.\]
By Lemma \ref{lem:c2k-count}, there exists a $j\in [h-1]$ such that the number of $C_{2k}$'s in $G$
that contain a vertex in $L_j(x)$ is at least 
\[\alpha_k |L_h| d_A^{k-h+j}\geq \alpha_k n^{h/k} (n^{1/k})^{k-h+j}=\alpha_k n^{1+\frac{j}{k}}.\] 
Let us denote this $j$ value by $j(x)$. 
 For each $t\in [h-1]$, let $S_t=\{x\in V(G): j(x)=t\}$. By the pigeonhole principle,
 for some $t\in [h-1]$, we have $|S_t|\geq n/(h-1)$. Let us fix such a $t$.
 By our discussion, for each $x\in S_t$, the number of $C_{2k}$'s that contain a vertex in $L_t(T_x)$
 is at least $\alpha_k n^{1+\frac{t}{k}}$. On the other hand, a vertex $y$ lies in $L_t(T_x)$ for at most $[\lambda D n^{1/k}]^t$
 different $x$. Hence the number of distinct $C_{2k}$'s in $G$ is at least
 \[|S_t| \alpha_k n^{1+\frac{t}{k}}/\lambda^t D^t n^{\frac{t}{k}}\geq (\alpha_k/k\lambda^k D^k) n^2.\]
 The claim holds by setting $\beta=\alpha_k/k\lambda^k D^k$.
\end{proof}

Now we can prove the $r=2$ case of Theorem \ref{thm:main}.

\medskip

{\bf Proof of the $r=2$ case of Theorem \ref{thm:main}:} Theorem \ref{thm:almost-regular-count} applies along as $D\geq 64k^2$,
$\lambda\geq 1$, and $n_0\geq (8\lambda)^k$. To apply Corollary \ref{cor:linear-cycle-reduction}, 
we set $D=\max\{64k^2, m_k\}, \lambda=q_kD$ and $M_k=(8\lambda)^k$,
where $m_k, q_k$ are as  given in Corollary \ref{cor:linear-cycle-reduction}. 
The claim follows readily from Corollary \ref{cor:linear-cycle-reduction}. \qed

\section{Supersaturation of even linear cycles in linear hypergraphs}\label{sec:hypergraphs}
For the supersaturation of linear cycles, we follow the approach of the $r=2$ case. However, instead of using the usual BFS tree, we need an adaption of it to hypergraph case. We define the notion of \emph{maximal rainbow rooted tree} in  Section 5.2.

\subsection{Notation and Preliminary Results} \label{supersat-prelim}

Let $H$ be a graph and $S$ be a some set of vertices, where possibly $S\cap V(H)\neq \emptyset$.
Let $\varphi$ be any colouring of the edges of $H$ using  non-empty subsets of $S$.
Given any subgraph $F$ of $H$, we let 
\[\mathcal{C}(F)=\bigcup_{e\in E(F)} \varphi(e),\]
and call it the {\it colour set of $F$ under $\varphi$}. We say
that $\varphi$ is {\it strongly proper} on $H$  if for any $e,e' \in E(H)$ that share a vertex we have $\varphi(e)\cap \varphi(e') = \emptyset$.  
We say that  $\varphi$ is  {\it rainbow} on $F$  (or that $F$ is {\it rainbow under $\varphi$}) if for every two edges $e,e'$ in $F$
we have $\varphi(e)\neq \varphi(e')$ and that 
$\mathcal{C}(F)$ is disjoint from $V(F)$.  Note that if $\varphi$ uses
 $(r-2)$-subsets of $S$ and $F$ is rainbow under $\varphi$ then $F\cup \mathcal{C}(F)$ forms an $r$-expansion of $F$.
Observe that  if $G$ is an $r$-partite linear $r$-graph with an $r$-partition $(A_1,\dots, A_r)$
then the \emph{natural} colouring $\varphi$ of $P_{i,j}(G)$, where $\forall f\in E(P_{i,j}(G)) \, \varphi(f)$ is the unique $(r-2)$-tuple
$I_f$ for which $f\cup I_f\in E(G)$, is strongly proper on $P_{i,j}(G)$ by the linearity of $G$.

Let $G$ be an $r$-graph  and  $v\in V(G)$. Recall the definition of 
$L_{G}(v)$ from the introduction. For any subset $S\subseteq V(G)$ we denote by $L_{G}(v)|_S$ the restriction of the link of $v$ to $S$, that is,
\[L_{H}(v)|_S = \{I \subseteq S | I \in L_H(v)\}.\]

We give a very crude analogue of Lemma \ref{lem:path-count}, this time counting rainbow paths of
a given length in an asymmetric bipartite graph. 

\begin{lemma} \label{lem:rainbow-path-count}
Let $p,m$ be positive integers and $H$ be a bipartite graph with a bipartition $(A,B)$. Let $\varphi$ be a strongly proper edge-colouring 
of $H$ using $m$-sets. If $\delta_A, \delta_B\geq 4p(m+1)$ then the number of rainbow paths of length $2p+1$ in $H$ is at least $\frac{1}{2^{2p}}\brm{e}(H)\left(\delta_A\delta_B\right)^p$.
\end{lemma}
\begin{proof}
Consider growing a rainbow path $P=v_1v_2\dots v_{2p+2}$ where $v_1\in A$ and $v_{2p+2}\in B$. There are $\brm{e}(H)$ choices for
$v_1v_2$. In general, suppose the subpath $v_1v_2\dots v_t$ has been grown, where $2\leq t\leq 2p+1$. If $v_t\in A$ then we let $v_{t+1}$ be a neighbor
of $v_t$ in $B$ such that $\{v_{t+1}\}\cup \varphi(v_t v_{t+1})$ 
is disjoint from  $((V(P_t)\setminus\{v_t\})\cup \mathcal{C}(P_t)$. If $v_t\in B$, $v_{t+1}$ is defined symmetrically.
Assume first that $v_t \in A$.
Note that 
\[|(V(P_t)\setminus \{v_t\})\cup \mathcal{C}(P_t))|\leq t-1+(t-1)m\leq 2p(m+1).\]
Since $\varphi$ is strongly proper, the set $\{u\cup \varphi(u): u\in N_H(v_t)\}$ is an $(m+1)$-uniform matching of size $d_H(v_t)$.
At most $2p(m+1)$ of these members contain a vertex in $(V(P_t)\setminus \{v_t\})\cup \mathcal{C}(P_t)$. So there
are at least $d_H(v_t)-2p(m+1)\geq \delta_A-2p(m+1)\geq \frac{1}{2}\delta_A $ choices for $v_{t+1}$. Similarly, if $v_t\in B$,
there there are at least $\frac{1}{2}\delta_B$ choices for $v_{t+1}$. Hence, the number of ways to grow $P$ is at least
\[\brm{e}(H)(\frac{1}{2} \delta_A)^p (\frac{1}{2}\delta_B)^p=\frac{1}{2^{2p}}\brm{e}(H)(\delta_A\delta_B)^p.\]
\end{proof}

\begin{lemma}[Splitting Lemma]\label{lem:splitting}Suppose we are given $D\in \mathbb{R}^+$, $\gamma \in (0,1)$ and integers $k,r\geq2$. There exists $n_0=n_{\ref{lem:splitting}}(D,k,r,\gamma)$ such that for all $n\geq n_0$  if $G$ is a linear $r$-partite $r$-graph such that two of its $r$-partition classes, say $A$ and $B$, satisfy that $|A\cup B|=n$ and that $|L_G(v)| \geq Dn^{\gamma}$ for each $v\in A\cup B$ then there exists a partition of $V(G)$ into $S_1, S_2, \dots, S_k$ such that for every $v\in A\cup B$ and every $i\in [k]$, we have
$$|L_{G}(v)|_{S_i}| \geq \frac{Dn^{\gamma}}{2 k^{r-1}}.$$
\end{lemma}
\begin{proof} 
Let us independently assign each vertex $x$ in $V(G)$ a colour from $[k]$ chosen uniformly at random.  Let $S_i$ be the vertices of assigned colour $i$. For a vertex $v\in A\cup B$, we denote by $X_i(v)$  the number of edges (which are $(r-1)$-sets) in $L_G(v)$ that are completely contained in $S_i$. For each $I\in L_{G}(v)$
\[\mP[I\subseteq S_i] = \frac{1}{k^{r-1}}.\]
Since $G$ linear, edges in $L_G(v)$ are pairwise disjoint. Hence the events  $\{I\subseteq S_i\}$, for different $I\in L_{G}(v)$'s are independent.
Therefore $X_i(v)$ has binomial distribution $\brm{BIN}(d_G(v), \frac{1}{k^{r-1}})$. Writing $d$ for $d_G(v)$, we have 
$\mE(X_i(v))=\frac{d}{k^{r-1}}$. By the Chernoff bound,
\[\mP[X_i(v) < \frac{d}{2 k^{r-1}}] \leq P[|X_i(v)- \frac{d}{ k^{r-1}}| < \frac{d}{2 k^{r-1}}]<2e^{-\frac{d}{12k^{r-1}}}<2e^{-\frac{Dn^{\gamma}}{12k^{r-1}}}.\]

Therefore the probability that for some vertex $v\in A\cup B$ and some $i\in [k]$ such that  the event $\{X_i(v) < \frac{d}{2 k^{r-1}}\}$ occurs is less than
$$kn\cdot 2e^{-\frac{Dn^{\gamma}}{12k^{r-1}}}<1,$$
when $n_2$ is large enough and  $n\geq n_2$. Thus there exists some colouring which guarantees for every vertex  $v\in A\cup B$ to have
$$|L_{G}(v)|_{S_i}| \geq \frac{d}{2 k^{r-1}} \geq  \frac{D n^{\gamma}}{2 k^{r-1}}.$$
\end{proof}

Before we establish supersaturation of $C^{(r)}_{2k}$'s in linear $r$-partite $r$-graphs that have
an almost regular $2$-projection with the right density, we need another lemma.
Given an $r$-graph $G$, where $r\geq 2$, and $S\subseteq V(G)$, $S$ is a {\it vertex cover} of $G$ if $S$
contains at least one vertex of each edge of $G$. 

\begin{lemma} \label{lem:cross-cut}
Let $r\geq 2$. Let $G$ be an $r$-graph and $S$ a vertex cover of $G$. 
There exist a subset $S'\subseteq S$ and a subgraph $G'\subseteq G$ such
that $\brm{e}(G')\geq \frac{r}{2^r} \brm{e}(G)$ and that $\forall e\in E(G') \, |e\cap S'|=1$.
\end{lemma}
\begin{proof}
Let $S'$ be a random subset of $S$ obtained by including each vertex of $S$ randomly and independently with probability $\frac{1}{2}$.
Let $e$ be any edge of $G$. Suppose $|e\cap S|=m$. Then $1\leq m\leq r$. The probability that exactly one of these $m$ vertices
of $e\cap S$ is chosen for $S'$ is $\frac{m}{2^m}\geq \frac{r}{2^r}$. So the expected number of edges of $G$ that meet $S'$
in exactly one vertex is at least $\frac{r}{2^r}\brm{e}(G)$. So there exists $S'\subseteq S$ such that at least $\frac{r}{2^r}\brm{e}(G)$ of the edges
meet $S'$ in exactly one vertex. Let $G'$ be the subgraph of $G$ consisting of these edges.
\end{proof}


\subsection{Rainbow Rooted Trees}
We now introduce the following adaption of the BFS tree to linear hypergraphs.

\begin{definition}[Maximal rooted rainbow tree] \label{def:expansiontree} Given $r\geq 3$, let $G$ be a linear $r$-partite $r$-graph  with  two of its partition classes being $A$ and  $B$ and let $t\geq 0$  be integer.  Suppose there exists a partition of $V(G)$ into $S_1, S_2, \dots, S_t$ such that for every  $v\in A\cup B$ and for every $i\in [t]$,
\begin{equation}\label{eq:nontrivial} L_G(v)|_{S_i} \neq \emptyset
\end{equation} For every $x\in A\cup B$, we define a  tree $T_x$, rooted at $x$ and of height $t$, together with a colouring $\varphi$ of its edges by $(r-2)$-sets as follows. We define the tree by defining its levels $L_i$ iteratively. The $L_i$'s will alternate between being completely inside $A$ and being completely inside $B$.
Without loss of generality, suppose $x\in A$. The tree $T_x$ is defined symmetrically if $x\in B$.
\begin{itemize}
\item [(1)]  Let $L_0=\{x\}$.
\item [(2)] Having defined $L_i$, we define $L_{i+1}$ as follows. Without loss of generality, suppose $L_i\subseteq A$. 
Let $F=\bigcup_{v\in L_i}{L_G(v)|_{S_{i+1}}}$. Since $G$ is $r$-partite with $A,B$ being two partite sets and $L_i\subseteq A$,
$V(F)$ is disjoint from $A$ and each edge of $F$ contains exactly one vertex in $B$.
Let  $M_{i+1}$ be a maximum matching in $F$. By condition \eqref{eq:nontrivial}, $M_{i+1}$ is nonempty.  Define 
\[L_{i+1}= M_{i+1}|_B= \{ b\in B|  \exists I \in M_{i+1} \textit{ such that } b\in I\}.\]
It remains to define how the vertices of $L_i$ are connected to $L_{i+1}$. For each $b\in L_{i+1}$, there exists a unique $I_b\in M_{i+1}$ which contains $b$,
and due to linearity of $G$ there is a unique $v\in L_i$ such that  $I_b\cup\{v\}\in E(G)$. We add the edge $vb$ to $T_x$ and let $\varphi(vb)= I_b\setminus\{b\}$.
\item [(3)] Repeat step (2) until all vertices of $H$ are exhausted or $i>t$.
\end{itemize}
\end{definition}

\begin{proposition}
Under the assumptions of Definition \ref{def:expansiontree}, $T_x$ is a tree of height $t$ rooted at $x$ that is rainbow under the assigned colouring $\varphi$.
In particular, if $P$ is a path in $T_x$ then $P\cup \mathcal{C}(P)$ is a linear path of the same length in $G$ with $P$ being a skeleton of it.
\end{proposition}
\begin{proof}
That $T_x$ is a height $t$ tree rooted at $x$ is clear from the definition. We now show that $T_x$ is rainbow under $c$.
By the way we define $T_x$ and $c$, $\mathcal{C}(T_x)\cap V(T_x)=\emptyset$.
Let $e,e'$ be any two edges in $T_x$. Suppose $e$ joins a vertex in $L_i$ to $L_{i+1}$ and $e'$ joins a vertex in $L_{i'}$ to
$L_{i'+1}$. If $i\neq i'$, then $\varphi(e)\cap \varphi(e')=\emptyset$, since $\varphi(e)\subseteq S_{i+1}$ and $\varphi(e')\subseteq S_{i'+1}$ and $S_{i+1}\cap S_{i'+1}=\emptyset$.
If $i=i'$ then $e\subseteq I$ and $e'\in I'$ for two different members $I,I'\in M_{i+1}$. Since $M_{i+1}$ is a matching, $\varphi(e)\cap \varphi(e')=\emptyset$.\
So $T_x$ is rainbow under $c$.

The second statement follows immediately from our discussion in Subsection \ref{supersat-prelim}
that a rainbow subgraph $F$ together with it colours form an expansion of $F$.
\end{proof}

We are now ready to prove the following analogue of Lemma \ref{lem:c2k-count}.
As we mentioned in the introduction, we will give a slightly different proof from that
of Lemma \ref{lem:c2k-count}. Instead of using Lemma \ref{balanced-root}, we will use
the strong/weak level notion used by Faudree and Simonovits \cite{FS} in the
study of theta graphs. Let us remark that we could also prove Lemma \ref{lem:c2k-count-hyper}
using Lemma \ref{balanced-root} and Lemma \ref{lem:rainbow-path-count}. But we feel that there is also a benefit
to use the strong/weak level notion used by Faudree and Simonovits since this is the
original approach we used to solve the problem and also that it is on some level more intuitive.

\begin{lemma}\label{lem:c2k-count-hyper}
Let $i,k,m$ be integers where $k\geq i+1\geq 1$, $m\geq 1$. Let $b,d$ be reals satisfying $b,d \geq 16^i(2m+2)k$. 
Let  $T_x$ be a tree of height  $i$ rooted at $x$.
For each $j=0,\dots, i$, let $L_j$ be the set of vertices in $T_x$ at distance $j$ from $x$. 
Let $W$ be some set of vertices disjoint from $V(T)$ and $H$ be a  bipartite graph with
bipartition $(L_i, W)$ such that  
\[\brm{e}(H)\geq \max\{d|L_i|, b|W|\}.\]
Suppose $c$ is an edge colouring of $G=T_x\cup H$ such that $c$ is rainbow on $T_x$ and
strongly proper on $H$ and that $\mathcal{C}(G)\cap V(G)=\emptyset$ and $\mathcal{C}(T_x)\cap \mathcal{C}(H)=\emptyset$.
Then there exist $0\leq q\leq i$ and some positive real $a_i=a_i(i,k)$
such that there are at least $a_i (bd)^{k-i-1+q} \brm{e}(H)$ many  rainbow $C_{2k}$'s in $G$ that contain a vertex in $L_q$.
\end{lemma}

\begin{proof} 
We proceed by induction on the height $i$ of the tree. It holds vacuously for $i=0$.  For all $i\geq 1$ we prove the result by splitting the argument into two cases and only in one of the cases we use induction. It is important to point out that when $i=1$ we are in Case 1 and thus need not use the vacuous case  of $i=0$ as our induction hypothesis.   

Let us denote by $x_1, x_2, \dots, x_p$ the children of $x$ in $T_x$. For each $j\in [p]$, let $T(x_j)$ be the subtree of $T_x$
rooted at $x_j$. For each $j\in [p]$,
we define the $j$th \emph{sector} to be $S_j= L_i\cap V(T_j)$. Note that  since $T_x$ is a tree, the $S_j$'s are pairwise disjoint. For a vertex $v\in L_i$, we denote by $S(v)$ the sector that $v$ lies in.

We say that  a sector $S_j$ is \emph{dominant} for a vertex $w\in W$ if 
\[|N_H(w)\cap S_j|> \max\left\{|N_H(w)| -2km,\frac{|N_H(w)|}{2}\right\}.\]
We say that $w\in W$ is \emph{strong} if it has no dominant sector and \emph{weak} otherwise.  Note that by our definition if $w\in W$ has a dominant sector
then there is only one such dominant sector for $w$.

Let $W_s$ be the set of strong vertices and $W_w$ be the set of weak vertices, respectively. Let $H_s$ denote the subgraph of $H$ induced by $L_i$ and $W_s$, $H_w$ denote the subgraph of $H$ induced by $L_i$ and $W_w$. The argument splits into two cases, depending whether 
the majority of the edges of $H$  lie in $H_s$ or in $H_w$. In the first case, we build the necessary number of
rainbow $2k$-cycles going through the vertex $x$ (so in the outcome of the theorem  we have $j=0$ as $x\in L_0$). In the second case we use induction to find  
rainbow $2k$-cycles in $T(x_j)$'s for many $j$. \medskip

\textbf{Case 1.}  \begin{equation}\label{edge-conditions} e(H_s)\geq \brm{e}(H)/2.\end{equation}

Let $d_{avg}(L_i)$ and $d_{avg}(W_s)$ denote the average degrees in $H$ for vertices in $L_i$ and $W_s$ respectively.
Then by \eqref{edge-conditions}, we have \(d_{avg}(L_i)\geq \frac{d}{2}, \quad d_{avg}(W_s)\geq \frac{b}{2}.\) By Lemma~\ref{lem:highmindegree}, there is a subgraph $H'$ of $H_s$ with bipartition $(A,B)$, $A\subseteq L_i$, $B\subseteq W_s$, such that
\begin{equation} \label{H'-bounds}
\brm{e}(H') \geq \frac{\brm{e}(H_s)}{2}\geq \frac{\brm{e}(H)}{4},\quad \delta_{A}(H') \geq \frac{d_{avg}(L_i)}{4} \geq  \frac{d}{8}, \quad \delta_{B}(H')\geq \frac{ d_{avg}(W_i)}{4} \geq \frac{b}{8}.
\end{equation}
Since $b,d\geq 16^i(2m+2)k$, clearly $\frac{b}{8}, \frac{d}{8}\geq (4m+4)k\geq (4m+4)(k-i-1)$.
Since $c$ is strongly proper on $H$, by Lemma \ref{lem:rainbow-path-count} with $p=k-i-1$,  the number of rainbow paths 
of length $2(k-i)-1$ in $H'$ is at least
\[\frac{1}{2^{2(k-i-1)}} e(H')(\delta_A(H')\delta_B(H'))^{k-i-1}\geq \frac{1}{2^{5(k-i-1)+2}} \brm{e}(H) (bd)^{k-i-1}.\]

\begin{claim} \label{path-to-cycle}
Every rainbow path $P=v_1v_2\dots v_{2(k-i)}$ of length $2(k-i)-1$ extends to a rainbow $C_{2k}$ in $G$ that contains $x$.
\end{claim}
{\it Proof of Claim.} By symmetry, we may assume that $v_1\in A, v_{2(k-i)}\in B$. For convenience, let $t=2(k-i)$.
It suffices to show that there exists $u\in N_H(v_t)$ (note that $u$ lies in $L_i$ but does not necessarily lie in $A$) 
such $P\cup v_t u$ is a rainbow path in $H$ and that $S(v_1)\neq S(u)$.
Indeed, suppose such $u$ exists. Then since $S(v_1)\neq S(u)$ the unique path $Q_1$ in $T_x$ from $v_1$ to $x$ and
the unique path $Q_2$ from $u$ to $x$ intersect only at $x$. Since $T_x$ is rainbow, $Q_1\cup Q_2$ is rainbow.
By our assumption, $\mathcal{C}(T_x)\cap \mathcal{C}(H)=\emptyset$. Thus, $P, Q_1, Q_2$ together form a rainbow $C_{2k}$ in $G$. 

Now we show that such $u$ exists. 
Since $v_t\in W_s$, by definition, $|N_H(v_t)\setminus S(v_1)|\geq 2km$.
Since $\varphi$ is a strongly proper edge-colouring using $m$-sets, $\{w\cup \varphi(w): w\in N_H (v_t)\setminus S(v_1)\}$ 
is an $(m+1)$-uniform matching of size $|N_H(v_t)\setminus S(v_1)|\geq 2km$.
Since clearly $|V(P)\cup \mathcal{C}(P)|<2km$, there exists $u\in N_H(v_t)\setminus S(v_1)$ such that 
$(w\cup \varphi(w))\cap (V(P)\cup \mathcal{C}(P))=\emptyset$. It is easy to see that $P\cup v_tu$ is a rainbow path in $H$.
Also, $u\notin S(v_1)$ by choice. \qed

\medskip

\textbf{Case 2:} $e(H_w)\geq \brm{e}(H)/2$.

\medskip
In this case, we have
\begin{equation}\label{hw-bounds}
e(H_w)\geq \frac{d}{2}|L_i|, \quad e(H_w)\geq \frac{b}{2}|W|.
\end{equation}
Recall that $x_1,\dots, x_p$ are the children of the root $x$ and for each $j\in [p]$, $S_j=V(T(x_j))\cap L_i$.  
For each $j\in [p]$, let $W_j$ be the set of  vertices in $W_w$ whose dominant sector is $S_j$.  Now we run the following ``cleaning" procedure. For every vertex $y\in W_w$ we only keep those edges in $H_w$ joining $y$ to vertices in its dominant sector. Let $H''$ denote the resulting subgraph of $H_w$.
By the definition of  $W_w$, every vertex $y\in W_w$ satisfies
\[d_{H''}(y)\geq |N_H(y)| - 2km.\]

Hence,
\[\brm{e}(H'') \geq e(H_w)  - 2km  |W_w|.\]

Since $c\geq 8km$, by \eqref{hw-bounds} $e(H_w)\geq 4km|W|$. Therefore
\begin{equation} \label{H''-bounds}
e(H'')\geq \frac{1}{2} e(H_w)\geq \frac{1}{4}\brm{e}(H).
\end{equation}

 For each $j\in [p]$, let $H_j$ denote the subgraph of $H''$ induced by $S_j\cup W_j$.
Note that the $H_j$'s are pairwise vertex-disjoint.  We want to  apply induction to those $T(X_j)\cup H_j$ where $H_j$ is relatively dense
from both partite sets. For that purpose we partition the index set $[p]$ as follows. Let
\[\cI_1=\{j\in [p]:  e(H_j)\leq \frac{d}{16}|S_j|\}, \quad 
\cI_2=\{i\in [p]: e(H_j) \leq \frac{b}{16}|W_j|\}, \quad
\cI_3=[p]\setminus (\cI_1\cup \cI_2).\]
By the definition and disjointness of  the $H_j$'s, we have 
\[\sum_{j\in \cI_1\cup \cI_2} e(H_j)\leq \frac{d}{16} |L_i| +\frac{b}{16}|W|\leq \frac{1}{8}\brm{e}(H).\]

Hence,
\begin{equation} \label{I3-bounds}
\sum_{j\in \cI_3} e(H_j)\geq \frac{1}{8}\brm{e}(H).
\end{equation}

For each $j\in \cI_3$, by definition, we have $e(H_j)\geq \frac{d}{16}|S_j|$ and $e(H_j)\geq \frac{b}{16}|W_j|$.
Since $T(x_j)$ has height $i-1$ and $\frac{d}{16}, \frac{b}{16}>(16)^{i-1}(2m+2)k$, by the induction hypothesis with $d,b$ replaced with
$\frac{d}{16}$ and $\frac{b}{16}$ respectively, there exists $q=q(j)$ such that the number of rainbow $2k$-cycles in $T(x_j)\cup H_j$ that
contain a vertex in level $q(j)$ of $T(x_j)$ 
is at least 
\[a_{i-1} \left(\frac{bd}{16^2}\right)^{k-(i-1)-1+q(j)} e(H_j)=a_{i-1} \left(\frac{bd}{16^2}\right)^{k-i+q(j)} e(H_j).\]

For each $t=0,\dots, i-2$, let $\cI_{3,t}=\{j\in \cI_3: q(j)=t\}$. By the pigeonhole principle, there exists  $t\in \{0,\dots, i-2\}$,
such that
\[\sum_{j\in \cI_{3,t}} e(H_j)\geq \frac{1}{i-1}\sum_{j\in \cI_3} e(H_j)\geq \frac{1}{8k} \brm{e}(H).\]

Let us fix such a $t$. By our earlier discussion and the fact that vertices in level $t$ of each $T(x_j)$ for $j\in \cI_{3,t}$
lie in level $t+1$ of $T_x$, the number of rainbow $2k$-cycles in $G$ that contain a vertex from  $L_{t+1}$ is at least
\[\sum_{j\in \cI_{3,t} }a_{i-1} \left(\frac{bd}{16^2}\right)^{k-i+t} e(H_j)= a_i (bd)^{k-i-1+(t+1)} \brm{e}(H),  \]
with the choice of $a_i= \frac{a_{i-1}}{2^{8(k-i+l) +3}k}$. Hence, in this case the lemma holds for $q=t+1$.
\end{proof}


\subsection{Proof of the $r\geq 3$ case of Theorem \ref{thm:main}}

We are finally ready to prove the supersaturation statement of $C^{(r)}_{2k}$ for
linear $r$-partite $r$-graphs $G$ that have a $2$-projection on two parts $A,B$ that is almost regular and have number of edges exactly $\Theta(|A\cup B|^{1+1/k})$. By Corollary~\ref{cor:linear-cycle-reduction} this would imply Theorem~\ref{thm:main} for all $r\geq 3$. For this we first define an adequate partition $V(G)$ into $S_1,\dots, S_k$.
From each vertex $x$ we define the maximal rainbow tree $T_x$ rooted at $x$ relative to the partition $(S_1,\dots, S_k)$. Then we apply
Lemma \ref{lem:c2k-count-hyper} to find many rainbow $2k$-cycles containing a vertex from some fixed level of $T_x$, which corresponds to linear
$2k$-cycles in $G$. Summing over all $x$ and
eliminating overcount, we get a lower bound on the number of $2k$-cycles in $G$.

\begin{theorem} \label{thm:almost-regular-hyper} Let $k, r\geq 2$ be integers. 
Let $D$ be a constant such that $D\geq 2^{r+1}rk^r(16)^k$. There exist $n_0$ such that if $G$ is a linear $r$-partite $r$-graph with an $r$-partition
$A_1,\dots, A_r$  such that  $|A_1\cup A_2| = n \geq n_0$ and  for every $v\in A_1\cup A_2$,
\[Dn^{1/k}\leq |L_G(v)| \leq \lambda D n^{1/k},\] where $\lambda\geq 1$ is a real, then there exists $\alpha =\alpha (k,r, \lambda)$ such that  $t_{C_{2k}^{(r)}}(G)\geq \alpha n^2$.
\end{theorem}
\begin{proof} The choice of $\alpha$ will be specified at the end of the proof.
We will choose $n_0$ be large enough so that $n_0\geq n_{\ref{lem:splitting}}(D,k,r,1/k)$, where $n_{\ref{lem:splitting}}$ is specified in Lemma \ref{lem:splitting}.
Let $S_1, S_2,\dots, S_k$ be a partition obtained by applying Lemma~\ref{lem:splitting} to $G$. In particular, for each $x\in A_1\cup A_2$ and $j\in [k]$, 
we have
\begin{equation} \label{deg-in-each-level}
|L_G(x)|_{S_j}|\geq \frac{Dn^{1/k}}{2k^{r-1}}.
\end{equation}

For each $x\in A_1\cup A_2$, let $T_x$ be a maximal rainbow tree of height $k$ rooted at $x$ relative to the partition $S_1,\dots, S_k$, as described
in Definition \ref{def:expansiontree}. The proof is similar to that of Theorem \ref{thm:almost-regular-count}. 
For each $x$, we find an $i\in [k]$ such that

\begin{itemize} \item [(i)] there exists some set $W'$ and  a bipartite subgraph $H_x$ induced by $L_i$ and $W'$ which
has high average degree from both partite sets.
\item[(ii)] The colouring $c$ on $T_x$ is extended to also include a strongly proper edge-colouring of $H_x$ such that
$\mathcal{C}(H_x)\cap \mathcal{C}(T_x)= \emptyset$.
\end{itemize}
We then use Lemma \ref{lem:c2k-count-hyper} to find many rainbow $2k$-cycles that contain some vertex in some fixed level of $T_x$.
Below are the details. 

Fix $x$ and write $T$ for $T_x$.  For $j=0,\dots, k$, let $L_j$ be defined as in Definition \ref{def:expansiontree}
and let $\varphi$ be the assigned edge-colouring of $T$ given in Definition \ref{def:expansiontree}.
Since $|L_1|\geq Dn^{1/k}>n^{1/k}$ and $|L_k|\leq n$, there exists a smallest $i\in [k-1]$ such that for all $1\leq j\leq i $, $|L_j| >  n^{j/k}$ but 
 \[|L_{i+1}| \leq n^{(i+1)/k}.\]
 
Let $T'$ be the subtree of $T$  induced by $\bigcup_{j=0}^i L_j$.
Let $F=\bigcup_{v\in L_i}{L_G(v)|_{ S_{i+1}}}$.  Since $S_{i+1}$ is disjoint from $S_1\cup\dots\cup S_i$ 
and since $L_i\subseteq A_1$ where $A_1$ is a partite set in an $r$-partition of $G$, $V(F)\cap V(T')=\emptyset$.
 By the construction of $T$, $|L_{i+1}|$ is equal to the size of a maximum matching in $F$. Since $F$ is an $(r-1)$-graph,
 we have   $\tau(F)\leq (r-1)\alpha'(F)$,
 where $\tau(F)$ and $\alpha'(F)$ denote the vertex cover number and matching number of $F$, respectively.
 Let $W$ be a minimum vertex cover of $F$. Then
\[|W| \leq (r-1)|L_{i+1}| \leq (r-1)n^{(i+1)/k} .\]
By Lemma \ref{lem:cross-cut}, there exist $W'\subseteq W$ and $F'\subseteq F$ such that

\[ e(F')\geq \frac{r-1}{2^{r-1}} e(F) \mbox{ and } \forall e\in e(F') \, |e\cap W'|=1.\]

 We define a bipartite graph $H_x$ between $L_i$ and $W'$ and extend the edge-colouring $\varphi$ restricted on $T'$ to an edge-colouring
 of $T'\cup H_x$ as follows.
 We go through the edges of $F'$ one by one.
 For each $e\in E(F')$, since $G$ is linear, there is a unique $v\in L_i$ such that $v\cup e\in E(G)$. Also by our definition of $F'$,
 $e\cap W'$ has exactly one vertex $w$. We include $vw$ in $H_x$ and let $\varphi(vw)= e\setminus\{w\}$. 
 By the linearity of $G$ and our discussion so far, each edge of $F'$ yields a different edge of $H_x$. 
 There is a bijection between $E(F')$ and $E(H_x)$. 
Moreover, $\mathcal{C}(H_x)\cap \mathcal{C}(T') = \emptyset$, since colours used on $H_x$ are $(r-2)$-sets in $S_{i+1}$ while
$\mathcal{C}(T')\subseteq S_1\cup \dots  \cup S_i$.

Since $G$ is linear and $r\geq 3$, $\forall v,v'\in L_i$ we have $L_G(v)|_{S_{i+1}}\cap L_G(v')|_{S_{i+1}}=\emptyset$. By \eqref{deg-in-each-level},
\[e(F)=\sum_{v\in L_i} |L_G(v)\cap S_{i+1}| \geq \frac{D n^{1 /k}}{2 k^{r-1}} |L_i|.\]

Hence, we have
\begin{equation}\label{eq:Hx-bound}
e(H_x)=e(F')\geq \frac{r-1}{2^{r-1}}e(F)\geq \frac{D(r-1)}{2^r k^{r-1}}n^{1/k}|L_i|.
\end{equation}

Also, by our choice of $i$, $|L_{i+1}|\leq n^{1/k}|L_i|$. Recall also that $|W'|\leq |W|\leq (r-1)|L_{i+1}|$. Hence,

\begin{equation} 
e(H_x)\geq  \frac{D(r-1)}{2^rk^{r-1}}|L_{i+1}|\geq \frac{D}{2^rk^{r-1}} |W|\geq\frac{D}{2^rk^{r-1}} |W'|.
\end{equation}

Let  $b=\frac{D}{2^rk^{r-1}} $ and  $d=\frac{D(r-1)}{2^r k^{r-1}}n^{1/k}$.
Since $D\geq 2^{r+1}rk^r(16)^k$,
we have  $d>b \geq  (2(r-2) +2) k (16)^k$.  So $T'$ and $H_x$ satisfy 
the conditions of Lemma~\ref{lem:c2k-count-hyper} with constants $b, d$ and $m= r-2$.
By Lemma \ref{lem:c2k-count-hyper}, there exists some $q=q(x)$ with $0\leq q\leq i$ and some  
$a_i=a_i(i,k)>0$ such that there are at least 
\[a_i(bd)^{k-i-1+q} e(H_x)\]
many rainbow $C_{2k}$'s in $T'\cup H_x$ that contain some vertex in level $L_q$ of $T'$. 
Now, $|L_i|\geq n^{\frac{i}{k}}$ by definition, $e(H_x)\geq \Omega(n^{\frac{i+1}{k}})$ by \eqref{eq:Hx-bound}. Also, $d=\Omega(n^{\frac{1}{k}})$. Hence,
the number of rainbow $C_{2k}$'s in $T'\cup H_x$ that contain some vertex in $L_q$ is at least
\[\beta n^{\frac{k-i-1+q}{k}}\cdot n^{\frac{i+1}{k}}=\beta n^{\frac{k+q}{k}},\]
for some $\beta=\beta(k,r)>0$.
So in $G$ there are at least $\beta n^{\frac{k+q}{k}}$ different linear $2k$-cycles
each of whose skeletons contains some vertex in $L_q$.

For each $t\in [k-1]$, let $S_t=\{x\in V(G)|\, q(x)=t\}$. By the pigeonhole principle, for some $t\in [k-1]$, $|S_t|\geq n/(k-1)$.
Let us fix such a $t$. Let $M$ denote the number of triples $(C,x,y)$, where $x\in S_t$, $C$ is a linear $2k$-cycle in $G$
whose skeleton contains a vertex in $L_t(T_x)$ and $y$ is a vertex on the skeleton of $C$ that lies in $L_t(T_x)$. Let $\mu$ denote the number of different linear
$2k$-cycles $C$ in $G$ that are involved in these triples.  By our discussion above, for each $x\in S_t$,
there are at least $\beta n^{\frac{k+q}{k}}$ different $C$. For each such $C$ there is at least one $y$. So
\begin{equation} \label{eq:M-lower}
M\geq |S_t|\beta n^{\frac{k+t}{k}}>(\beta/k) n^{2+\frac{t}{k}}.
\end{equation}

On the other hand, for each of the $\mu$ linear $2k$-cycles $C$ involved, there are at most $2k$ different choices of $y$.
For fixed $y$, there are at most $(\lambda D n^{1/k})^t$ choices of $x$ since such
$x$ is at distance at most $t$ from $y$ in the $(1,2)$-projection $P_{1,2}(G)$ of $G$, which has maximum degree at most $\lambda D n^{1/k}$. So,
\begin{equation} \label{eq:M-upper}
M\leq \mu (2k) (\lambda D n^{1/k})^t.
\end{equation}
Combining \eqref{eq:M-lower} and \eqref{eq:M-upper} and solving for $\mu$, we get
\[\mu\geq \frac{\beta}{2k^2 (\lambda D)^t} n^2.\]
Let $\alpha=\frac{\beta}{2k^2 (\lambda D)^k}$. Then $\alpha$ is a function of $k,r,\lambda$ and we have
$t_{C^{(r)}_{2k}}(G)\geq \mu\geq \alpha n^2$.
\end{proof}

We are now ready to prove the $r\geq 3$ case of Theorem \ref{thm:main}.

\medskip

{\bf Proof of the $r\geq 3$ case of Theorem \ref{thm:main}:} 
First note that Theorem \ref{thm:almost-regular-hyper} can be rephrased as saying that if
$G$ is linear $r$-partite $r$-graph that has a $2$-projection $P$ on at least $m\geq n_0$ vertices such that $Dm^{1/k}\leq \delta(P)\leq \Delta(P)\leq \lambda Dm^{1/k}$
then $t_{C^{(r)}_{2k}}(G)\geq \alpha m^2$.
The statement holds  as long as $D\geq 2^{r+1}rk^r(16)^k$,
$\lambda\geq 1$, and $m\geq n_0$. To apply Corollary \ref{cor:linear-cycle-reduction}, 
we set 
\[D=\max\{2^{r+1}rk^r(16)^k, m_k\}, \lambda=q_kD, \mbox{ and } M=\max\{n_0, m_k\},\]
where $m_k, q_k$ are as  given in Corollary \ref{cor:linear-cycle-reduction}. 
The claim follows readily from Corollary \ref{cor:linear-cycle-reduction}. \qed


\section{Concluding remarks}\label{sec:conclusion}
 
First, we say a few words about the difference between our proofs between the $r=2$ and the $r\geq 3$ cases for almost regular host graphs.
The one for $r=2$ uses Lemma \ref{balanced-root} and the one for $r\geq 3$ uses induction on the height the maximal rainbow tree.
As we pointed out both proofs work for both cases. We choose to present one for each to illustrate both methods.
The one that avoids induction potentially could be applied in other settings such as when considering odd linear cycles.

Next, we would like to point out that the reduction to proving the supersaturation of $C_{2k}$ for $n$-vertex host graphs
$G$ with density exactly at $\Theta(n^{1+1/k})$ is crucial to the proof of our general theorem. Using the BFS approach, one
can indeed find many copies of $C_{2k}$. However, the approach works perfectly only when $G$ has density $\Theta(n^{1+1/k})$.
For denser $G$, one can still get a bound, but the bound becomes worse and worse compared to the optimal $c(\frac{\brm{e}(G)}{\brm{v}(G)})^{2k}$
as $G$ gets denser. A reason for that is the subgraph of $G$ induced by consecutive levels of a BFS tree is now much denser
and no longer resembles a tree structure. If we only use the BFS structure to construct our $C_{2k}$'s, we will lose count on many
$C_{2k}$'s. It might be possible to make the BFS approach work directly for dense $G$ without a reduction. 
But the analysis becomes exceedingly complicated.

For all integers $k,p\geq 2$, the theta graph  $\Theta_{p,k}$ is the graph consisting of $p$ many internally disjoint paths of length $k$ sharing the same endpoints. It was shown by Faudree and Simonovits \cite{FS} that $ex(n,\Theta_{p,k})=O(n^{1+1/k})$.
The method of our paper can be used to establish the  supersaturation of  the $r$-expansion $\Theta_{p,k}^{(r)}$ (where $r\geq 2$) of $\Theta_{p,k}$
in linear $r$-graphs. When $r=2$ this establishes the truth of Conjecture~\ref{bipartite-supersaturation} for $H=\Theta_{p,k}$ with $\alpha=\alpha'=1-1/k$. Again, the lower bound is tight up to a multiplicative constant, obtained by taking a random graph of an almost complete Steiner system.

It would be very interesting to establish the supersaturation of odd linear cycles in linear $r$-graphs, for $r\geq 3$. Toward this end, in \cite{CGJ} it is shown that when $r\geq 3$ we have $ex_l(n,C^{(r)}_{2k+1})=O(n^{1+1/k})$, which
is very different from the $2$-uniform case where for all sufficiently large $n$ it is known that $ex(n,C_{2k+1})=\lceil \frac{n}{2}\rceil\lfloor \frac{n}{2}\rfloor$.
The proof of this theorem is much more involved than its counterpart
for even linear cycles. It is unclear if a similar supersaturation statement as Theorem \ref{thm:main}  holds
for $C^{(r)}_{2k+1}$. At least our methods don't readily give this. We raise this as an open question.

\begin{question}
Let $k,r$ be integers where $k\geq 2, r\geq 3$. Do there exist positive constants $C$ and $c$ depending
only on $k$ and $r$ such that every $n$-vertex linear $r$-graph $G$ with $\brm{e}(G)\geq Cn^{1+1/k}$ contains
at least $c \left(\frac{\brm{e}(G)}{\brm{v}(G)}\right)^{2k+1}$ copies of $C^{(r)}_{2k+1}$?
\end{question}

Very recently, Balogh, Narayanan and Skokan~\cite{BNS} obtained a balanced supersaturation 
result for  linear cycles of all lengths in general $r$-graphs. Note that this is a different setting from ours, as  in our case host graphs are linear, and hence are sparse, while they are working with dense ones. As Morris and Saxton did for cycles in graphs, Balogh et al. used their supersaturation result to obtain a bound 
on the number of of $n$-vertex $C_m^{(r)}$-free  $r$-graphs. It would be interesting to  obtain such a balanced version of supersaturation for even linear cycles in linear $r$-graphs as well and hence generalize the result of Morris and Saxton. Our methods  a priori do not give such strong supersaturation.

Another problem worth exploring is to sharpen the result of Balogh et al. Unlike for $2$-uniform even cycles, when $r\geq 3$
the usual Tur\'an number of the $r$-uniform linear cycle $C^{(r)}_m$ has been completely determined in \cite{FJ} and \cite{KMV} 
for all sufficiently large $n$. Asymptotically, $ex(n, C^{(r)}_m)\sim \lfloor \frac{m-1}{2}\rfloor
\binom{n}{r-1}$. The supersaturation result of Balogh et al. applies to $n$-vertex $r$-graphs $G$ with $\brm{e}(G)\geq C\cdot ex(n,C^{(r)}_m)$ for a sufficiently large constant $C$.
It is possible that one can establish a similar statement for all $n$-vertex $r$-graphs $G$ with $\brm{e}(G)\geq (1+o(1)) ex(n,C^{(r)}_m)$ and hence
sharpen the bound on the number of $n$-vertex $C^{(r)}_m$-free $r$-graphs. As a supersaturation problem on its own without the application
to the count of $C^{(r)}_m$-free graphs,  it would also be interesting to at least  establish supersaturation 
of $C^{(r)}_m$ in all $n$-vertex $r$-graphs $G$ with $\brm{e}(G)\geq (1+o(1))ex(n,C^{(r)}_m)$.

\end{document}